\documentclass[a4paper,10pt]{article}

\RequirePackage{amsmath,amsthm,amsfonts,amssymb,enumerate,mathrsfs,euscript}
\RequirePackage[numbers,sort&compress]{natbib}
\usepackage[ruled,lined]{algorithm2e}
\usepackage{dsfont}
\RequirePackage[colorlinks,citecolor=blue,urlcolor=blue]{hyperref}
\RequirePackage{graphicx}
\usepackage{subcaption}
\usepackage{geometry}
\theoremstyle{plain}

\newtheorem{theorem}{Theorem}
\newtheorem{lemma}{Lemma}
\newtheorem{coro}{Corollary}
\newtheorem{prop}{Proposition}
\theoremstyle{remark}
\newtheorem{definition}{Definition}
\newtheorem{assumption}{Assumption}

\newcommand{\udl}{\underline}
\newcommand{\ovl}{\overline}
\newcommand{\ps}[1]{\langle #1\rangle}
\newcommand{\bP}{{\mathbb P}}
\newcommand{\bE}{{\mathbb E}}
\newcommand{\bR}{{\mathbb R}}
\newcommand{\bN}{{\mathbb N}}

\newcommand{\cC}{{\mathcal C}}

\newcommand{\cF}{{\EuScript F}}

\newcommand{\1}{{\mathds{1}}}

\renewcommand{\le}{\leqslant}
\renewcommand{\ge}{\geqslant}
\DeclareMathOperator{\leb}{Leb}
\newcommand{\reels}{\mathbb{R}}

\newcommand{\dr}{\,\mathrm{d}}
\newcommand{\ex}{\mathrm{e}}

\newcommand\ud[2]{\underline{#1}_{\,#2}}

\newcommand\p[1]{\left( #1 \right)}
\newcommand{\espcond}[2]{\mathbb{E}\mathopen{}\left[#1\middle|#2\right]}
\newcommand{\proba}{\mathbb{P}}

\newcommand\norm[1]{\left\lVert #1 \right\rVert}
\newcommand\abs[1]{\left\lvert #1 \right\rvert}
\begin{document}

\title{Stochastic mirror descent for \\nonparametric  adaptive importance sampling}
%\title{A sample article title with some additional note\thanksref{t1}}
%\runtitle{Stochastic mirror descent}
%\thankstext{T1}{A sample additional note to the title.}

%%%%%%%%%%%%%%%%%%%%%%%%%%%%%%%%%%%%%%%%%%%%%%%
%% Only one address is permitted per author. %%
%% Only division, organization and e-mail is %%
%% included in the address.                  %%
%% Additional information can be included in %%
% \author[A]{\fnms{Bianchi}~\snm{P.}\ead[label=e1]{pascal.bianchi@telecom-paris.fr}},
% \author[B]{\fnms{Delyon}~\snm{B.}\ead[label=e2]{bernard.delyon@univ-rennes.fr}},
% \author[C]{\fnms{Portier}~\snm{F.}\ead[label=e3]{francois.portier@gmail.com}}
% \and
% \author[A]{\fnms{Priser}~\snm{V.}\ead[label=e4]{victor.priser@telecom-paris.fr}}
% %%%%%%%%%%%%%%%%%%%%%%%%%%%%%%%%%%%%%%%%%%%%%%
% %% Addresses                                %%
% %%%%%%%%%%%%%%%%%%%%%%%%%%%%%%%%%%%%%%%%%%%%%%
% \address[A]{LTCI, Telecom Paris \printead[presep={,\ }]{e1,e4}}

% \address[B]{IRMAR, University of Rennes 1 \printead[presep={,\ }]{e2}}

% \address[C]{CREST, ENSAI \printead[presep={,\ }]{e3}}
\author{{Bianchi}~{P.},
{Delyon}~{B.},
{Portier}~{F.} and 
{Priser}~{V.}}
%%%%%%%%%%%%%%%%%%%%%%%%%%%%%%%%%%%%%%%%%%%%%%
%% Addresses                                %%
%%%%%%%%%%%%%%%%%%%%%%%%%%%%%%%%%%%%%%%%%%%%%%
\maketitle

\begin{abstract}

This paper addresses the problem of approximating an unknown probability distribution with density $f$ - which can only be evaluated up to an unknown scaling factor - with the help of a sequential algorithm that produces at each iteration $n\geq 1$ an estimated density $q_n$.
%The proposed method does not require gradient evaluations or knowledge of the normalization constant of the target density \(f\), making it applicable in a variety of contexts such as Bayesian inference, reinforcement learning, and stochastic optimization. 
The proposed method optimizes the Kullback-Leibler divergence using a mirror descent (MD) algorithm directly on the space of density functions, while a stochastic approximation technique helps to manage between algorithm complexity and variability. One of the key innovations of this work is the theoretical guarantee that is provided for an algorithm with a fixed MD learning rate  \(\eta \in (0,1 )\). The main result is that the sequence \(q_n\) converges almost surely to the target density \(f\) uniformly on compact sets. Through numerical experiments, we show that fixing the learning rate \(\eta \in (0,1 )\) significantly improves the algorithm's performance, particularly in the context of multi-modal target distributions where a small value of $\eta$ allows to increase the chance of finding all modes. Additionally, we propose a particle subsampling method to enhance computational efficiency and compare our method against other approaches through numerical experiments.

\end{abstract}

\section{Introduction}
% \vp{J'ai l'impression que c'est nécessaire d'introduire $f_u$ ici}

% \fp{je pense que le lecteur comprend et je souhaite introduire le moins de notation possible dans l'intro.}

   Consider the problem of approximating an unknown probability distribution with density $f : \mathbb R^d \to \mathbb R _{\ge 0} $ using a sequential algorithm that produces an estimated density $q_n$. The index $n\in \mathbb N$ here stands for the number of point-wise evaluations of $f$ - no evaluation of the gradient $\nabla f$ is needed- and the knowledge of the normalization constant of $f$ should not be necessary, i.e., 
   {the algorithm remains the same when, for an arbitrary constant $c>0$, $cf$ is used instead of $f$.} This framework is useful in many applications such as Bayesian inference \cite{engel2023bayesian,huang2023efficient} or variational inference \citep{blei2017variational}, reinforcement learning \cite{metelli2018policy,hanna2019importance} or stochastic optimization \citep{needell2014stochastic,guilmeau2024divergence} to name a few among the statistical learning literature.

The \textit{adaptive importance sampling} method \citep{ho+b:1992, owen+z:2000,bugallo2017adaptive} or the \textit{sequential Monte Carlo} approach \cite{del2006sequential}, are based on generating random variables according to a certain sampling distribution that evolves during the algorithm and using some re-weighting allows to obtain unbiased estimators. Depending on the problem of interest, the sampling distribution might be chosen by minimizing some discrepancy with respect to the target measure and many such different approaches have been investigated in \cite{elvira2022optimized,NEURIPS2019_ba7609ee,li2016renyi}. A leading approach, coming from the variational inference literature \citep{blei2017variational}, consists in minimizing the Kullback-Leibler (KL) divergence, defined as
$$\text{KL}(q \|f ) : =  \int \log( q / f  ) q , $$
with respect to $q$ chosen out of a parametric family of density functions. Throughout the paper, $\int h $ shall be used as a shortcut for $\int h(x) dx $. The optimization framework attached to the {variational inference} approach is attractive because of the recent development in stochastic optimization and related methods (e.g., stochastic gradient descent, variance reduction). This, in particular, allows to handle large scale problems as promoted in \cite{hoffman2013stochastic}. 
%For all three previous approaches, the approximation of the probability distribution of $f$ (and the related quantities such as mean and variance) is obtained by means of weighted particles $(\omega_{n} ,X_n)_{n\ge 1}$ where the $\omega_n$'s are some weights properly chosen depending on the method and where the $X_n$'s are random variables in $\mathbb R^d$ often called \emph{particles}. 
%\vp{C'est vraiment utile de parler de particules? comme maintenant c'est une autre définition pour moi}
%\fp{OK on enlève partout alors}
%\vp{j'avais pas vu qu'il y avait des particules en dehors de l'intro, je pensais à sample mais c'est quand même plus vendeur "particules"}

In \cite{dai2016provable,korba2022adaptive,chopin2023connection}, the \textit{mirror descent} (MD) algorithm \citep{beck2003mirror} is employed to optimize the \textit{Kullback–Leibler} divergence directly on the space of density functions and thereby avoiding parametric misspecification issues as in standard stochastic variational inference \citep{hoffman2013stochastic}. 
When applied to $\text{KL}(q \| f) $, the MD algorithm gives the following iteration, for $n\ge 0$,
\begin{align}\label{eq:MD_algo}
q_{n+1} \propto q_{n}^{1-\eta} f^{\eta }\,    \end{align}
 where $\eta\in (0,1] $ is called the \textit{learning rate}. The above iteration cannot be implemented within our framework simply because 
% {\color{red} $\int f^\eta$ is unknown and hard to compute} 
{$f$ is unknown (so is the normalizing constant in the above)}. 

The approach taken in this paper is to rely  on \textit{stochastic approximation} \citep{rob_mon_51} whose main idea is to resort to a sequential algorithm in which a computationally cheap 
 stochastic update is conducted at each step. 
 %Next is described how this general principle can be applied to approximate \eqref{eq:MD_algo}. 
 Suppose that $X_{n+1} $ is generated from $q_n $ and define the importance weights $w_{n+1} = f(X_{n+1}) / q_{n}(X_{n+1})     $. Let \(K_b:{\mathbb R}^d\to {\mathbb R}_{\ge 0}\) be a probability density with a mean zero and  covariance \(b^2 I_d\), where \(I_d\) is the identity matrix of size $d\times d$, and \(b>0\) is a small positive parameter known as the \emph{bandwidth}. Define the random map
$$M_{n+1}:x\in\bR^d\mapsto  w_{n+1}^\eta K_b ( x - X_{n+1} ) ,$$
and note that 
$\bE(M_{n+1})   = \{ q_{n}^{1-\eta} f  ^{\eta }\} \ast K_b , $
where \(\ast\) denotes the convolution product between functions, i.e., $f  *  g (x) = \int f(y)g(x-y)dy$ when $f,g$ are real-valued Lesbegue-integrable functions, which from well-known results from approximation theory should be near the MD iteration expressed in \eqref{eq:MD_algo}, $q_{n}^{1-\eta} f  ^{\eta }$, when $b$ is small. Having this in mind, the proposed algorithm follows from the functional iteration
\begin{equation*} 
g_{n+1} = (1-\gamma_{n+1}) g_n + \gamma_{n+1} \, M_{n+1} ,
\end{equation*}
where  \(\gamma_{n+1}\) is a positive step size converging to $0$ as \(n\) tends to infinity. The final step is given by \( q_{n+1} = (1-\lambda_{n+1}) g_{n+1}/\int g_{n+1}  + \lambda_{n+1} q_0\) where \(\int g_{n+1}\) can be easily determined through the algorithm and $q_0$ is a heavy-tailed distribution that ensures a comprehensive exploration of the space \({\mathbb R}^d\) as suggested by \cite{delyon2021safe}. The distribution $q_n$ writes as a mixture between $q_0$ and a weighted sum of $n$ kernels anchored at the particles locations. Such a sum of kernels is encountered in the wellknown context of  \textit{kernel density estimation} \citep{parzen1962}.

In the previous algorithm, the choice $\eta = 1$ might be attractive at first glance because of the bias term which is easy to analyze \citep{delyon2021safe}. However, when the dimension $d$ is large or when the function $f$ is complex, the importance weights often collapse \citep{bengtsson2008curse}, i.e., only a few weights carry-out the whole probability mass. This implies that the weights exhibit significant variance and this characteristic severely hampers the algorithm's efficiency, causing it to stagnate at excessively high weights. As noted by \cite{korba2022adaptive},
\[ 
\text{Var}\left( w_n^\eta  \right) \le \text{Var}\left(w_n \right) 
\]
for \(\eta\in(0,1)\), suggesting that choosing small $\eta$ might help to reduce the variability and avoid the degeneracy of the importance weights. In addition, as observed in practice, and similarly to several stochastic gradient descent optimization algorithm choosing a fixed stepsize allows the algorithm to explore the space of interest and thereby avoid local minima. %This idea is illustrated in Figure where a $2$-dimensional Gaussian mixture is considered. (F: write a bit more on the setup but not too much as this is still an introduction). 

\textbf{Related algorithms.} The proposed method has connections with several well-known approaches within the {adaptive importance sampling}, {sequential Monte Carlo} and {variational inference} literature. The idea of using a stochastic approximation of the MD iteration \eqref{eq:MD_algo} using an adaptive importance sampling approach was first investigated in \cite{dai2016provable} and further studied in \citep{korba2022adaptive}. In both previous papers, the importance weights $w_n$ are ``tempered'' using the MD power transform $w_n^\eta$ but their results are significantly different from the one established in this paper in that  \cite{dai2016provable} considers $\eta$ converging to $0$, and \cite{korba2022adaptive} deals with learning rate $\eta$ going to $1$, while here we study the case where $\eta $ remains fixed during the algorithm. Same weights transformations, referred to as \textit{tempering}, as well as other similar transformations e.g., \textit{clipping}, that implies a different behavior than standard weights, have been used in several {sequential Monte Carlo} algorithms \citep{neal2001annealed,koblents2015population,aufort2022tempered} (without relying specifically on kernel smoothing); see also \cite[Section 2.3.1]{del2006sequential}. For more detail on the connection between \textit{tempering} and MD, we refer to \cite{chopin2023connection}.
%\vp{je comprends pas d'ou sort le tempering, tu as commenté des lignes qui parlent de tempering}

The proposed algorithm bears resemblance to the ones of \cite{west1993approximating,Givens1996,zhang1996nonparametric,delyon2021safe} even though the MD iteration was not considered in the previous work (i.e., $\eta = 1$). In all previous references, a kernel smoothing estimator is employed to estimate $f$. This is also done by several popular {sequential Monte Carlo} samplers as described in \citep{chopin2004central,del2006sequential} where generating from normalized $g_n$, i.e., $X_{n+1} \sim g_n / \int g_n$, is often described using two steps: (i) the \textit{selection step} chooses at random, using multinomial sampling, one particular particle $X_i$ among the existing ones; (ii) the \textit{mutation step} generates $X_{n+1}$ around $X_i$ using (for example) kernel $K_b({\cdot} - X_i)$.

%one particular technique employed in SMC sampling, called \textit{tempering} and described in \cite{neal2001annealed}, allows to move from $q_0$ to $f$ slowly by using several iteration similar to  of the type $q_0 ^(1-\eta) f^\eta $ which is similar to the proposed algorithm. For more detail on the link between \textit{tempering} and MD, see \cite{chopin2023connection}.

Note that the proposed algorithm only requires evaluation of the target density up to a normalization constant. This differs significantly from approaches that rely on gradient evaluations as for instance the \textit{Markov chain Monte Carlo} methods proposed in  \citep{welling2011bayesian,simsekli2016stochastic} or the {variational inference} approaches detailed in \citep{liu2016stein,korba2020non,li2022sampling}.  %Consequently, it will be most pertinent to benchmark and compare our algorithm solely with other methods that similarly forgot the utilization of gradients.

\textbf{Contribution and related results.} The main result of the paper is to establish the almost sure convergence (uniformly on compact sets) of the sequence $(q_n)_{n\ge 0}$ to the target density $f$, under the assumption that the bandwidth tends to zero as $n$ tends to infinity \emph{i.e.}, $b=b_n$ vanishes, and that $\gamma_n$ satisfies typical Robbins-Monro conditions \citep{rob_mon_51}, while $\eta$ might be constant during the algorithm. One important consequence of the previous is a central limit theorem, with rescaling factor $\sqrt n $, for the estimation of $\int hf $, for compactly supported test functions $h: \mathbb R^d \to \mathbb R$, using empirical weighted average of $h(X_n)$ with weights $w_n$.

Existing theoretical results on the convergence of sequence $(q_n)_{n\ge 0}$ when $\eta= 1$ might be found in \citep{west1993approximating,Givens1996,zhang1996nonparametric,delyon2021safe}.
%\cite{zhang1996nonparametric} gives a convergence result to an integral of the type $\int fg$ where $f$ is the target density. Later works have introduced a "safe" heavy-tailed component $q_0$ into the target density estimate.  
For instance, the almost sure uniform convergence to the target density with a convergence rate and some central limit theorem for the integral estimation problem are obtained in  \cite{delyon2021safe}. 
In \cite{dai2016provable} some results concerning the weak convergence of $q_n$ to $f$ and the convergence of the \textit{Kullback-Leibler} objective are given when $\eta $ goes to $0$. In \cite{korba2022adaptive}, the almost-sure convergence is established when $\eta$ converges to $1$. 

 To the best of our knowledge, the results of the present paper are the first theoretical guarantees about the convergence of $q_n$ to $f$ when $\eta$ is fixed during the algorithm.
 The new parasitic stationary point at $0$ when $\eta<1$ complicates the proof as we need to establish that the algorithm is not trapped in the vicinity of this spurious equilibrium. In contrast, forcing $\eta$ to converge to $1$ as in \cite{korba2022adaptive} makes the algorithm behave asymptotically like the case $\eta=1$ as s studied in \citep{delyon2021safe}, which eases the proof. 

 As mentioned previously, similar types of algorithms, that involves  a power transformation of the weights $w_n^\eta$, often refereed to as tempering, have been studied within the sequential Monte Carlo literature \cite{chopin2004central,del2006sequential,douc2007limit}. To our knowledge, the results obtained are different as the evolution of the sequence $(q_n)_{n\ge 0}$ is stopped while allowing the number of particles to go to infinity \citep{chopin2004central}. This constraint is heavy because in practice one might allow the policy to change in time without constraint.

Another line of work is the nonparametric recursive estimation problem in which data is used sequentially to update the estimators \cite{devroye19801,mokkadem2009stochastic,bercu2019nonparametric}. Note that the recursive estimation framework relies on Robbins-Monro type procedure, just as we do, but the context is different because our framework requires the variational policy $(q_n)_{n\ge 0}$ to be updated and then used to draw points whereas in the recursive framework the data is distributed according to a fixed density.

\textbf{Outline.} In Section \ref{s2}, we the mathematical framework and the main algorithm along with several practical remarks. In Section \ref{s3}, we state our main result and provide a sketch of proof. In Section \ref{s4}, we consider several practical variations of the proposed method while in Section \ref{s5}, we compare them to other approaches using numerical experiments.  All the proofs are provided in the Appendix.

\section{The algorithm}\label{s2}

%\subsection{The algorithm}

%\subsection{Definition}

We consider a probability density function $f:\bR^d\to [0,+\infty)$, referred to as the target. Let $f_u:\bR^d\to [0,+\infty)$ be a Lebesgue integrable function, representing an unnormalized version of $f$. That is, there is a constant $c>0$ such that $f_u = cf$.
%We denote by $f := f_u/\int f_u$ its normalized version.
%The function $f$ is a probability density function, 

Let $(\Omega,\cF,\bP)$ be a probability space. Consider a sequence $(X_n)_{n\ge 1}$ of random variable on $\bR^d$.
Denote by $(\mathcal F_n)_{n\ge 0} $ the natural filtration associated to the sequence $(X_n)_{n\ge 1}$. That is, $\mathcal F_n  = \sigma (X_ 1, \dots , X_n)$
for $n \ge 1$, and $\mathcal F_0 = \{ \emptyset ,\Omega\} $. The sequence $(X_n)_{n\ge 1}$ is specified by its \textit{policy} $(q_n)_{n\ge 0}$ defined as follows.

\begin{definition}
  The sequence of random variable $ (q_n)_{n\ge 0}$ is said to be a \emph{policy} of $ (X_n)_{n\ge 1}$, if it is adapted to $(\cF_n)_{n\ge 0}$
  and if, for every $n\ge 0$, $q_n$ is a conditional probability distribution function of $X_{n+1}$ given $\cF_n$, that is, $\bE(h(X_{n+1})|\cF_n) = \int h q_n$
  for every bounded continuous function $h$ on $\bR^d$. 
\end{definition}

Define the \emph{importance weights} by
\begin{equation}
\label{eq:w}
w_{n+1} :=   \frac{f_u(X_{n+1})}{q_n(X_{n+1})}\,, \quad n\ge 1\,.    
\end{equation}
These weights play an important role in the importance sampling framework as they allow to shift the distribution from $q_n$ toward the target distribution $f$. As such, it allows to estimate without bias integrals with respect to the unnormalized target distribution as, whenever $q_n>0$ implies $f_u>0$, it holds
$$ \bE( w_{n+1} h(X_{n+1})|\cF_n) = \int h f_u\,. $$
We are now in a position to introduce our algorithm characterizing sequentially the policy $(q_n)_{n\ge 0}$ based on a sequence of unnormalized density $(g_n)_{n\ge 0}$. Let $K:\bR^d\to [0,\infty) $ be a probability density function and define $K_b(x) = b^{-d}K(x/b)$
for any $b>0$, the corresponding density with variance $b^2 I_d$.  At each step $n\geq 0$, the random variable $X_{n+1}$ is drawn from $q_n$ and the distribution $g_n$ is updated into $g_{n+1}$ as follows: 
\begin{align}
&X_{n+1} \sim q_n \nonumber\\
&   g_{n+1}(x) = (1-\gamma_{n+1}) g_n(x) + \gamma_{n+1} \, w_{n+1}^\eta K_{b_{n+1}} ( x - X_{n+1} ) \,, \qquad \forall  x\in \bR^d, \label{eq:g} 
\end{align} 
where $(\gamma_n)_{n\ge 1}$ and 
$(b_n)_{n\ge 1}$ are positive sequences, respectively referred to as the
step size and the bandwidth sequences. We set $g_0 = 0$. The next step is therefore to define $q_{n+1}$ from $g_{n+1}$, by:
\begin{align} \label{eq:q}
q_{n+1}(x) = (1-\lambda_{n+1}) \frac{g_{n+1}(x)}{\int g_{n+1}} +\lambda_{n+1}q_0(x)\,,\qquad \forall  x\in \bR^d,
\end{align}
where $q_0$ is a fixed probability density function, and $(\lambda_n)_{n\ge 1}$ is a positive sequence. Thus, we do not define $q_{n+1}$ as the normalized version of $g_{n+1}$, but as a mixture between the latter and a fixed distribution $q_0$. This mixture step will be revealed essential in our proofs, in order to ensure sufficient exploration, and thus prevent $q_n$ to converge to a spurious stationary point. The parameter $\lambda_{n+1}$ determines the tradeoff between the exploration and the adaptation to $f$.

%\subsection{Stylized facts}

The next proposition is given without proof as it is an easy consequence of the normalization stage in \eqref{eq:q}. 

\begin{prop}\label{prop:norm}
    The policy $(q_n)_{n \ge 0} $ obtained from \eqref{eq:q} is invariant with respect the normalization constant $c>0$. 
\end{prop}

The previous property is attractive because it implies that even when $f_u$ is attached to a small normalization constant $c>0$, it has no effect on the algorithm (even in the first iterations). This fact, in our proof, will be useful as it will allow us to work directly with the true target density $f$.

\begin{figure}[ht!]
    \centering
    \begin{subfigure}[b]{0.48\textwidth}
        \centering
        \includegraphics[width=0.9\textwidth]{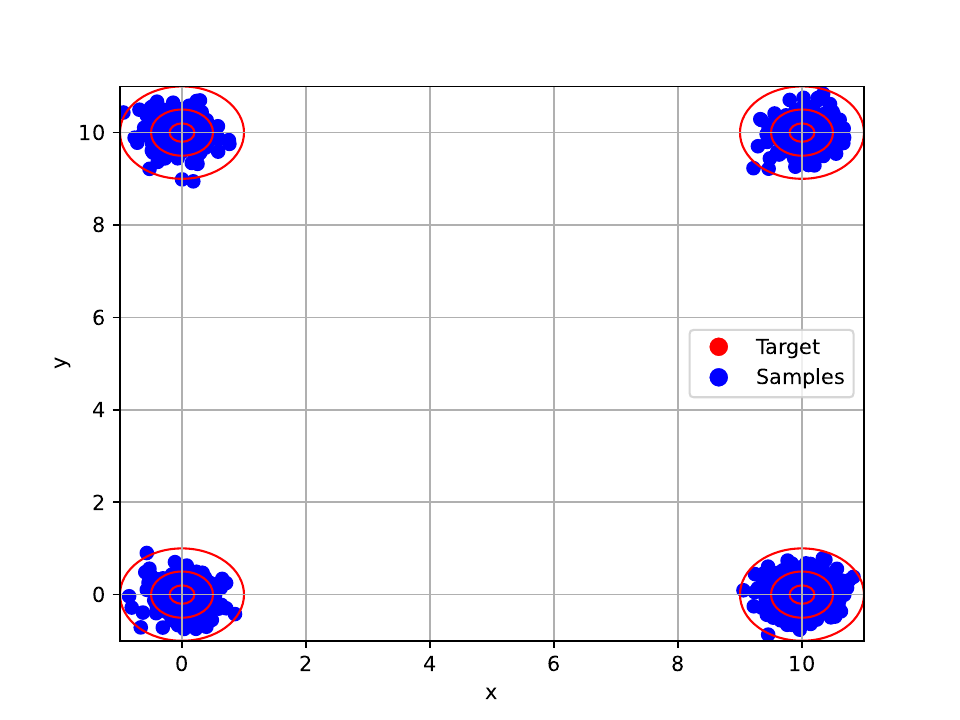}
        \caption{All the modes are found}
        \label{fig:conv}
    \end{subfigure}
    \hfill
    \begin{subfigure}[b]{0.48\textwidth}
        \centering
        \includegraphics[width=0.9\textwidth]{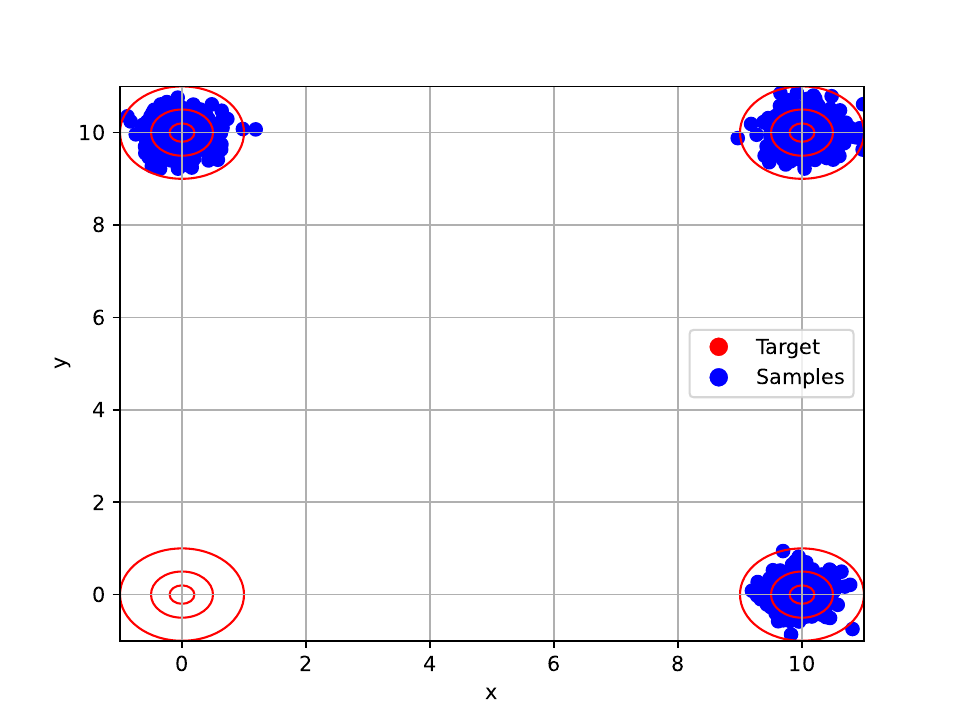}
        \caption{One mode is missing}
        \label{fig:nonconv}
    \end{subfigure}
    \begin{subfigure}[b]{0.9\textwidth}
        \centering
            \includegraphics[width=.9\textwidth]{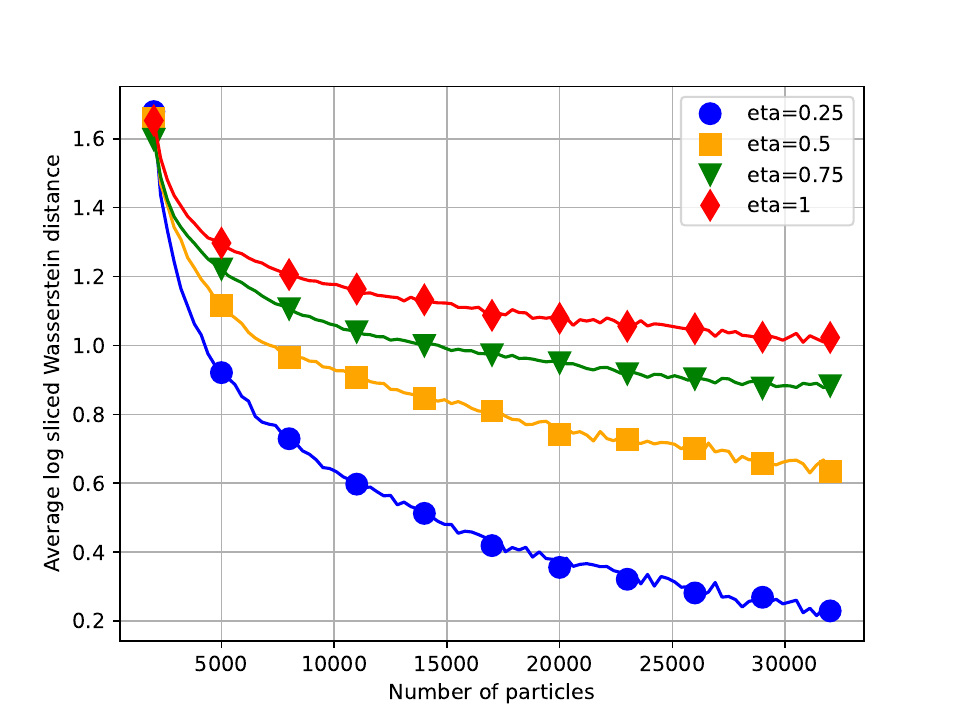}
        \caption{Average over $50$ independent runs of the logarithm of sliced Wasserstein distance for different value of $\eta$}
        \label{fig:2Dplot}
    \end{subfigure}
    \caption{The target is the multi-modal distribution described in the text. Figure \ref{fig:conv} represents samples generated from $q_n$ ($\eta= 1$) where all modes are identified while Figure \ref{fig:nonconv} has missed one mode. Figure \ref{fig:2Dplot} shows a comparison (based on the sliced Wasserstein distance) of different algorithms when varying $\eta$.}
    \label{fig:convornot}
\end{figure}

The role of $\eta $ has already been depicted as balancing between bias and variance in \cite{korba2022adaptive} as it was shown that $ 
\text{Var}( w_n^\eta  ) \le \text{Var} (w_n ) 
$ while $ \mathbb E [ w_n ^\eta] \leq 1$. One related point is it can also enable to visit extensively the domain of interest. To illustrate this claim, we now consider the extreme case where $\eta = 0 $, and we remark that in this case, the policy $(q_n)_{n \ge 0} $ obtained from \eqref{eq:q} does not depend on $f$.
%$ \gamma_n = 1/n$ and $\lambda _n = 0$. 
 The choice $\eta = 0$ is not expected to be efficient  as it does not use the evaluations of $f$. It is nonetheless informative as it stresses that, as soon as $\eta$ is small, there is no preference between the weights. As a consequence, a new particle tends to be drawn equally from any previous particle. This is particularly attractive as in some cases the algorithm can be trapped in a small part of the domain while missing other parts.

% \begin{figure}[h!]
%     \centering
%     \subfigure[Image 1]{\includegraphics[width=0.45\textwidth]{2Dconv.png}}
%     \label{subfig:conv}
%     \subfigure[Image 2]{\includegraphics[width=0.45\textwidth]{2DnonConv.png}}
%     \label{subfig:nonconv}
%     \caption{Description des deux images}
%     \label{fig:convornot}
% \end{figure}

To illustrate the previous property and thereby the importance of the parameter \(\eta\), we provide a toy example with a multi-modal target in dimension \(d=2\), where all modes are challenging to find. The target density is a mixture of four Gaussian distributions with means at \((0,0)\), \((10,0)\), \((0,10)\), and \((10,10)\), each having a variance of \(0.1 I_2\). The heavy-tailed density \(q_0\) is a Student's t-distribution with a location parameter of \((5,5)\) and a scale parameter of \(10 I_2\).% The parameter \(\lambda_n\) is set to 0. 
The target distribution possesses four modes. In Figure~\ref{fig:convornot}, we present two illustrative runs of our algorithm. In Figure~\ref{fig:conv}, all four modes of the target distribution are recovered while in Figure~\ref{fig:nonconv}, one mode is missing. A run of the algorithm with one mode missing can occur randomly for any value of \(\eta\), but by setting the parameter low enough, the algorithm is more likely to recover all the modes of the target distribution. To evaluate how much this impact the outcome of the algorithm, we compute the sliced Wasserstein distance, as defined in Section~\ref{sec:setup}, between the target distribution and the weighted empirical measure of the particles \((X_n,w_n)_{n \ge 1}\).
The average over $50$ independent runs of the log Wasserstein distance is  provided in
Figure~\ref{fig:2Dplot}.

\section{Main results}\label{s3}

\subsection{Almost sure convergence}\label{s3Sub}

We start by giving the assumptions needed on the target $f_u$ and the density $q_0$.

\begin{assumption} \label{hyp:fq}\ 
  \begin{enumerate}[i)]
  \item \label{hyp:fq:reg}The functions $f_u$, $q_0$, are bounded, continuous, nonnegative, and integrable on $\bR^d$. 
%  \item The sequence $(g_n)$ satisfies~(\ref{eq:main}) for all $n\ge 0$ along with $g_0:=q_0$, and is an unnormalized policy of $(X_n)$.
  \item \label{hyp:fq:fbound} There exists $r>0$ and $C_f>0$ such that $\int_{\|x\|>t}f_u(x)dx \le C_ft^{-r}$ for all $t\ge 0$.
  \item     \label{hyp:fq:f<q0}
  There exists $c>0$ such that $cf_u\le  q_0$.   
  \end{enumerate}
\end{assumption}

Assumption~\ref{hyp:fq}-\ref{hyp:fq:fbound} holds for instance if the function 
$x\mapsto \norm{x}^{r+d}f_u(x)$ is bounded in $\reels^d$. 
%The proposed algorithm explores the space with the help of  $q_0$ and the sequence $(\lambda_n)_{n\ge 1}$. 
Assumption~\ref{hyp:fq}-\ref{hyp:fq:f<q0} ensures that the support of the target distribution $f_u$ is included in the support of $q_0$, which will implies (assuming that $\lambda _n $ is large enough) that $f_u$ can be explored thoroughly. %To ensure a good exploration, we want to maximize the quantity
%$\inf_{x,f(x)\neq 0} \frac{q_0(x)}{f(x)}$, which is greater than a constant $c$. 
This condition is necessary for the definition of the weights $(w_n)_{n\ge 1}$.
 %It means that $q_0$ must be tuned with respect of the density we want to estimate. 
We also need to have some regularity and integrability conditions on the \textit{kernel} function $K$.

  \begin{assumption}\label{hyp:K} \ 
  \begin{enumerate}[i)]
      \item \label{hyp:K:reg}$K:\mathbb R^d\to \mathbb R_{\ge 0}$ is a bounded and Lipschitz density function and $K(0)>0$.
      \item \label{hyp:K:bound} There exist $r>0$ and $C_K>0$ such that, for all $x\in \mathbb R^d$, $(1+\|x\|^{d+r})K(x)\le  C_K$.
    \end{enumerate}
\end{assumption}
The  {Gaussian kernel}, $K(u) \propto \exp(- \|u\|^2 / 2 ) $, or the Epanechnikov Kernel, $K(u ) \propto (1- \|u\|^2 )_+ $ where $u_+= u$ when $u \ge 0$ and $0$ else, satisfy the above assumption. In the numerical experiments, the Gaussian kernel will be used.
%Let \(A \subset \mathbb{R}^d\). The function \(\1_A: \mathbb{R}^d \to \mathbb{R}\) denotes the indicator function, which is equal to 1 when \(x \in A\) and 0 when \(x \notin A\).
Finally we state the assumption needed on the step-size sequence $(\gamma_n)_{n\geq 1}$, bandwidth sequence $(\gamma_n)_{n\geq 1}$ and mixture parameter $(\lambda_n)_{n\geq 1}$.
 % \label{hyp:rate}
 % The three positive sequences $(\gamma_n,b_n,\lambda_n)_{n\ge 1}$ satisfy the following.
 % \begin{enumerate}[i)]
 % \item \label{hyp:rate:gamma} There exists $0<\alpha\le 1$
 %   such that
%$$
%\gamma_n = O(n^{-\alpha}) \text{ and } \gamma_n \ge \tfrac{1}{n} \text{ for $n$ large enough}\, .
% 
%    \[
%        b_n ^{-1} = O(n^\beta)\text{ and } b_n\to 0\, .
%    \]
%\item \label{hyp:rate:lambda} There exists $\kappa>0$
%    such that
%$$
%{\lambda_n}^{-1} = O(n^{\kappa})\text{ and } \lambda_n \to 0\,.
%$$
%\item \label{hyp:rate:coef} The following conditions hold:
%\begin{gather}    \label{eq:condition-beta-kappa} 
%   1+\kappa\big(2-2\eta+(2\eta-1)_+\big)+d\beta\big(1+(1-2\eta)_+\big)<2\alpha \\
%   \frac {n\gamma_{n}^2}{\lambda_n b_n^d } \ll 1
%   \pb{~~~(On~ peut~ simplifier~ avec:(2\eta-1)_+\le\eta, ~~1+(1-2\eta)_+\le 2-\eta)}\nonumber
%\end{gather}
%\end{enumerate}
%  \end{assumption}

\begin{assumption}
  \label{hyp:rate}\ 
  \begin{enumerate}[i)]
\item  \label{hyp:rate:all_sequence}  The sequence $(\gamma_n) _{n\ge 1}$, $(b_n) _{n\ge 1}$, $(\lambda_n) _{n\ge 1}$ are decreasing to $0$.
  \item \label{hyp:rate:gamma} The exists $C_\gamma >0$ and  $n_\gamma\ge 1$ such that  for all $n\ge 1$,
$  n^{-1}\le C_\gamma \gamma_n $, and, for all $n\ge n_\gamma$, $ \gamma_{n} - \gamma_{n+1} \le \gamma_n\gamma_{n+1}$. Moreover, 
$$ \sum_{n\ge 1} \gamma_n^2 <\infty .$$
%$$ n^{-1} \le \gamma_n \le C n^{-\alpha} \quad \text{and} \quad  \gamma_ n  -  \gamma_{n+1} \le C\gamma_n  \gamma_{n+1} $$
\item \label{hyp:rate:coef}
When  $  \eta\ge 1/2$, it holds that $$   \frac{n\gamma_{n}^2 \log n}{\lambda_n b_n^d } \to  0.$$ 
When $ \eta < 1/2$, we have $$\frac{n\gamma_{n}^2\log n}{\lambda_n^{2(1-\eta)} b_n^{2d(1-\eta)}} \to  0.$$
%   \pb{~~~(On~ peut~ simplifier~ avec:(2\eta-1)_+\le\eta, ~~1+(1-2\eta)_+\le 2-\eta)}\nonumber

% \item\label{hyp:rate:RS}  $$\text{when }  \eta\ge \frac 12, \quad\sum_{n\ge 2} \frac{\lambda_n b_n^{d}}{n\log n } <\infty \,.$$

\end{enumerate}
  \end{assumption}
% $$ \sum_i^n \psi_i^{n2}\gamma_i^2 \mathbb E [ W_i^2] \le n \gamma_n ^2 \lambda_n ^{1-2\eta}$$
%   and
% $$ \sum_i^n \psi_i^{n2}\gamma_i^2 \mathbb E [ W_i^2 K_b ] \le n \gamma_n ^2 \lambda_n ^{1-2\eta} b_n^{-d}$$
  
We remark, that Assumption~\ref{hyp:rate}-\ref{hyp:rate:gamma} %and \ref{hyp:rate}-\ref{hyp:rate:gamma} 
holds when $\gamma_n = Cn^{-\alpha}$, for $C>0$ and $\alpha \in (1/2,1)$, or $C>1$ and $\alpha =1$.
%Possible choices include $\gamma_n = C n^{-1} $ with $C\ge 1$ or $\gamma_n = C n^{-\alpha} $ with $\alpha \in (1/2,1)$ and $ C> 0$. 
When $\gamma_n = C n^{-1}$ and $\eta>1/2$,  Assumption  \ref{hyp:rate}-\ref{hyp:rate:coef} on ($b_n$) implies a classical condition in the kernel smoothing estimation literature, that is, \[
\lim_{n\to\infty} nb_n^d =\infty \text{ and } \lim_{n\to\infty} b_n =0\, .
\]
In non parametric estimation when the function is at least 2-times continuously differentiable and the kernel has order 2 \citep{10.1214/aos/1176345206}, the optimal bandwidth is
$  h_n =   n^{ - \frac{1}{4+d}}$. This choice is made possible by the assumption of our main result. 
%because we have considered the general case where there is a trap in zero but in the special case $\eta=1$, there is no trap and our condition is stronger than necessary.

Based on the previous set of assumptions, we are able to prove the almost sure convergence on compacts. This is the main result of the paper.

  \begin{theorem}
    \label{th:main}
    Consider the policy $(q_n)_{n\ge 0}$ given by~(\ref{eq:q}).
    Let Assumptions~\ref{hyp:fq},~\ref{hyp:K} and ~\ref{hyp:rate} hold true and let  $A\subset \reels^d$ be  a compact set. Then, almost surely, we have:
\[
 \lim_{n\to\infty} \sup_{x\in A}\abs{q_n(x) -f(x)} = 0\,.
\]
% Moreover, we have $Z_n\to \p{\int f_u}^{\eta}$ almost surely.\vp{peut être supprimer cette dernnière partie si onn ne pale pas de l'integrale}
\end{theorem}

\subsection{Sketch of the proof}\label{sketch}

By Proposition~\eqref{prop:norm}, we can consider $f=f_u$.
We can rewrite \eqref{eq:g} as follows
\begin{align*}
 g_{n+1} &=  (1-\gamma_{n+1}) g_{n} + \gamma_{n+1} q_n^{1-\eta} f^\eta \ast K_{b_{n+1}} 
 + \gamma_{n+1}\xi_{n+1} \,,
\end{align*}
where $\xi_{n+1} : \mathbb R^d \to \mathbb R$ is a martingale increment in that, for all $x\in \mathbb R^d$, $\mathbb E [\xi_{n+1} (x) |\mathcal F_n] = 0 $. From the latter equation, we see that the algorithm has two equilibria: $f$ and $0$. Note that the case $\eta = 1$ does not present this issue, which is why our proof differs significantly from those in \cite{delyon2021safe} and \cite{korba2022adaptive}. The point $f$ is stable, while the point $0$ is unstable. This means that, without the martingale term involving $(\xi_n)_{n\ge 1}$, $(g_n)_{n\ge 0}$ would converge to $f$, but the presence of the martingale term can potentially cause it to get trapped at $0$. 

Using that $q_n = g_n/(\int g_n) + \lambda_n q_0$, a convexity inequality implies that
\begin{align*}
 g_{n+1} \ge  (1-\gamma_{n+1}) g_{n} + \gamma_{n+1} \frac {(1-\lambda_n)^{1-\eta}}{2^\eta(\int g_n)^{1-\eta}} T (g_n) \ast K_{b_{n+1}}& \\
 + \gamma_{n+1}\xi_{n+1} + \gamma_{n+1} \frac{\lambda_n }{2^\eta}  T(q_0) & \,,
\end{align*}
where $T(q) = q^{1-\eta} f^\eta$. By iterating the latter equation, we obtain a lower bound on $g_n$. Using a Freedman-type inequality, we show that, for sufficiently large $n$, the term involving the heavy tailed density $q_0$ is greater than the absolute value of the martingale term involving $(\xi_n)_{n \ge 1}$ in the expression of this lower bound.
Hence, cases where the martingale term could lead the algorithm into the trap cannot occur. The trap is therefore avoided. With this in hand, the analysis of convergence becomes more straightforward.

% In order to show that the trap in $0$ is avoided, we introduce a lower bound $(g_{n})_{n\ge 0}$ of $(g_n)_{n\ge 0}$ such that $\udl g_{\,0} = 0$ and 
% \begin{align*}
%  g_{n+1} &=  (1-\gamma_{n+1}) g_{n} + \gamma_{n+1} q_n^{1-\eta} f^\eta \ast K_{b_{n+1}} 
%  + \gamma_{n+1}\xi_{n+1} + \gamma_{n+1}\lambda_n (\tfrac{f(x)}{2})^\eta (q_0(x))^{1-\eta} \,,
% \end{align*}
% with a sequence of martingale increments $(\xi_n)_{n\ge 1}$.
% For every $n\ge 1$, this can be rewritten as
% \[
% \udl g_{\,n} = v_n + M_n + \tilde \lambda _nq_0\, ,
% \]
% where $v_n$ depends on $(\udl g_{\,i})_{0\le i\le n}$, the martingale terms $(M_n)_{n\ge 1}$ depends on $(\xi_n)_{n\ge 1}$ and $(\tilde \lambda_n)_{n\ge 1}$ depends on $(\lambda_n)_{n\ge 1}$.

% Then, we use a Freedman type concentration inequality to show that the sequence of martingale terms $(M_n)_{n\ge 1}$ is dominated by the positive sequence $(\tilde \lambda_n)_{n\ge 1}$ for large $n$. This means that, with high probability, and for large $n$, $\udl g_{\,n}$ is greater than $v_n$. With a deterministic method, we show that $v_n$ is greater than $f$, which gives that $g_n \ge f$.

Note that in \cite{korba2022adaptive, portier2018asymptotic}, a Freedman-type concentration inequality is also used for the analysis of the martingale term.  When $\eta \in [1/2, 1)$, we recover the same condition as in 
\cite{korba2022adaptive, portier2018asymptotic} while when $\eta \in (0,1/2) $, the condition in Assumption~\ref{hyp:rate}-\ref{hyp:rate:coef} is stronger. This is due to the variance of the noise that does not scale the same way. %We believe that the assumption obtained in this case might be alleviated.

\subsection{Convergence in total variation and weak convergence.}

An application of Scheffé's lemma allows to extend the uniform convergence on compact sets to $L_1$-convergence.

  \begin{coro}\label{cor:L1}
Under the assumptions of Theorem~\ref{th:main}, we have, almost surely, 
\[
    \lim_{n\to\infty} \int \abs{q_n -f} = 0 .
\]
  \end{coro}

 Now we can turn our attention to weak convergence type of results for the estimation of integrals. This property has some practical interest in regards of the Bayesian application where one is often interested in computing the mean with respect to posterior distribution. With the help of Algorithm 1, integral of the form $\mu(h) = \int h f$, for a given integrable function $h$, can be easily estimated using the normalized quantity 
 $$ 
\hat \mu_n (h) =  \frac  {\sum_{i=1}^n w_i h(X_i) }{\sum_{i=1}^n w_i}.
 $$
 The asymptotic normality is established in the next proposition.

\begin{coro}\label{coro:TCL}
    Let $h: \bR^d \mapsto \bR $ with compact support $A$ such that $\int fh^2 < \infty$ and suppose that $\inf_{x\in A} f(x) > 0$. Under the assumptions of Theorem~\ref{th:main}, 
    $
\sqrt n (\hat \mu_n (h)  -  \mu (h)) \leadsto \mathcal N ( 0, \sigma^2(h)) 
    $ 
with $ \sigma^2 (h) = \mu(h^2) - \mu(h) ^2$.
\end{coro}

The proof of this result, which is given in the Appendix, follows from an application of the Lindeberg central limit theorem \cite{hall2014martingale} with a careful use of the convergence of $(q_n)_{n\ge 1}$ in order to check each of the conditions leading to the right asymptotic variance. We note that the expression of the asymptotic variance is the same as the one of the oracle Monte Carlo estimate $(1/n) \sum_{i=1} ^n h(X_i)$ where $(X_i)_{i\ge 1} $ is an independent sequence of random variables with common distribution $f$. This equicontinuity property is reminiscent of Corollary 1 stated in \cite{portier2018asymptotic} where a high-level condition is given on $q_n$ to satisfy such a central limit theorem.

\section{Practical considerations}\label{s4}In this section, we provide a complete description of the considered algorithms including subsampling and minibatching variants.

\subsection{Initial algorithm} To present a concise description of the studied algorithm, let us start with some algebra expanding~\eqref{eq:g} and \eqref{eq:q} which together provide the incremental expression of the algorithm,
as a damped stochastic version of mirror descent. Assuming $g_0=0$ and taking $\lambda_0= 1$, the policy $(q_n)_{n\ge 0}$ writes:
\begin{equation}\label{eq:newq}
    q_n(x) = (1-\lambda_n)\p{\frac{\sum_{i=1}^n W_{i,n} K_{b_i}(x-X_i)}{\sum_{i=1}^n W_{i,n}}} + \lambda_n q_0(x)\,  ,
\end{equation}
where for $1\le i\le n $, 
\begin{equation}
    \label{eq:W}
    W_{i,n} :=  w_i^\eta \gamma_i\prod_{j=i}^n (1-\gamma_j) \,.
\end{equation} 
The practical implementation of \eqref{eq:newq} and \eqref{eq:W} is detailed in Algorithm \ref{alg1}, referred to as \textit{mirror descent for adaptive sampling} (MIDAS), for a budget $N\ge 1$ corresponding to the number of evaluation of $f$.

\begin{algorithm}[ht]\caption{MIrror Descent for Adaptive Sampling (MIDAS)}\label{alg1}
\KwIn{Budget $N\ge 1$, step sizes $(\gamma_n)_{1\le n\le N}$, bandwidths $(b_n)_{1\le n\le N}$, mixture weights $(\lambda_n)_{1\le n\le N}$, learning rate $\eta\in(0,1]$,  initial distribution $q_0$, kernel $K$}
\KwOut{Weighted particles $(X_n, W_{n,{N}})_{1 \le n\le N}$ }
%\nl generate $X_1$ from the mixture $q_{0}$ and compute $W_{1,1} \leftarrow \p{\frac{f_u}{q_{0}}(X_1)}^\eta \gamma_1$\;
\nl\For{$n\leftarrow 0$ \KwTo $N-1$}
{\nl generate $X_{n+1}$ from the mixture $q_{n}$ defined in \eqref{eq:newq}\;
\nl $W_{n+1,n+1}\leftarrow \p{\frac{f_u(X_{n+1}) }{q_{n}(X_{n+1}) }}^\eta \gamma_{n+1}$\;
\nl if $n \ge 1$, for all $1\le i\le n$, $W_{i,n+1} \leftarrow (1-\gamma_{n+1}) W_{i,{n}}$\;
}
\end{algorithm}

Thanks to \eqref{eq:newq}, it is easy to see that MIDAS is invariant with respect to the choice of $f_u$ among all possible scaled versions of the density $f$. As a result the algorithm does not require the knowledge of the normalizing constant but also the behavior of the algorithm is not sensible to the value of this constant. This makes easier the choice of the hyperparameter such as the learning rate $\eta$ or the step-size $\gamma_n$. 
An interesting special case is obtained when setting $\gamma_n = 1/n$. In this case, we obtain the simplification: $W_{i,n} = w_i^\eta/n$, for all $1\le i\le n$. If, moreover, $\eta=1$, the algorithm is closely related to the algorithm of \cite{delyon2021safe} even though we consider here a slightly different choice of the kernel's bandwidth sequence.

     Sampling from the mixture $q_n$ in \eqref{eq:newq} is achieved by drawing a random index in $\{1,\dots,n\}$ with a probability equal to the weights $W_{i,n}$ for $1\le i\le n$. The cost of generating the index is $O(\log n)$. Next, a random variable is generated according to the kernel density $K$, which yields the final particle $X_{n+1}$ (up to shifting and rescaling). Of course, the kernel $K$ is chosen to make the latter step computationally effective.

 As justified in the previous paragraph, we shall neglect the (logarithmic) cost of drawing a particle from the current distribution $q_n$.
Hence the main computational cost of the algorithm is carried out by the evaluation, at each iteration, of the importance weight $W_{n+1,n+1}$ and in particular to the computation of $f_u(X_{n+1} )$ and $q_{n}(X_{n+1} )$. We denote by $c_u$ the cost of evaluating $f_u$. The evaluation of $q_n(X_{n+1})$ requires $n$ evaluations of $K$. Denoting by $c_K$ the cost for evaluating $K$ at a given point, the $n$-th iteration of the algorithm requires an order of $c_u+ n c_K$. This leads to an overall computing cost of $ N c_u + N^2c_K +$. Even if in some practical situations (e.g., complex Bayesian model or when $f_u$ is the result of a heavy simulation program), $c_K$ is might be smaller than $c_u$, the complexity is dominated by the quadratic term $N^2 c_K$ when $N$ is large. In this case, it is interesting to consider a variant of our algorithm, in which the complexity is reduced. This is the purpose of the next section.

%The cost of evaluating a density or generating a random variable which we denote by $C$ is bigger than the cost of a simple operation. Moreover, the cost of evaluating $f_u$ that we denote $c_f$ may be greater than $C$ (e.g., Bayesian likelihood). That's why we dissociate it from $C$. The cost of computing $W_{n+1,n+1}$ is of order $c_f + C n $ as evaluating the density $q_n$ of a mixture of $n$ random variables consists of evaluating the density of $n$ random variables. The rest of the operations costs $n$. Thus, the cost of a step $n\ge 1$ of the algorithm is of order $Cn +c_f + n$. Then when the cost of evaluating $f_u$ is small, we want to reduce the cost of evaluating $q_n$ which is of order $Cn$. The next subsection will be dedicated to this purpose.

%We also need simulating according to $q_{i-1}$ which is of order $\log (i)$ because it mainly consists of simulating according to a multinomial distribution.
%\begin{itemize}
%    \item The cost of evaluating $w_i$ is of order $c_f + i$
%    \item The cost of simulating according to $q_{i-1}$ is of order $\log (i)$ because it mainly consists of simulating according to a multinomial distribution
%\end{itemize}

\subsection{Subsampling variant}

In order to decrease the quadratic in $N$ computing cost of MIDAS, we propose a subsampling version of MIDAS, which is inspired from \cite{Gordon_1993}. The aim is to restrict the iteration cost of the algorithm to approximately $O(\ell_n)$ operations, where $(\ell_n)_{n\ge 0}$ is a sequence of integers chosen by the user herself, and such that $\ell_n\ll n$. One may for instance consider $\ell_n\sim n^\delta$, with $0<\delta<1$.

At each iteration $n$, the main idea is to draw $\ell_n$ indices $(u({i,n}))_{1\le i\le \ell_n}$ according to the following weighted empirical distribution
$$
\proba_n := \frac{\sum_{i=1}^n W_{i,n} \delta_{i}}{\sum_{i=1}^n W_{i,n}}\,,
$$
where we recall the definition of $W_{i,n}$ in~\eqref{eq:W}.
In other words, at the $n$th iteration of the algorithm, $\ell_n$ particles with its bandwidth are drawn with replacement among the $n$ particles generated so far by the algorithm.  
Then, we define
\begin{equation}\label{eq:q*}
q_n^*(x) := (1-\lambda_n)\sum_{i=1}^{\ell_n}K_{b_{u\p{i,n}}}\p{x - X_{u\p{i,n}}} + \lambda_n q_0(\cdot)\,,
\end{equation}
The algorithm, which will be referred to as SubMIDAS for subsampling MIDAS, is described in Algorithm~\ref{alg2}.

\begin{algorithm}[ht]\caption{SubMIDAS}\label{alg2}
\KwIn{Budget $N\ge 1$, step sizes $(\gamma_n)_{1\le n\le N}$,  bandwidths $(b_n)_{1\le n\le N}$, mixture weights $(\lambda_n)_{1\le n\le N}$, learning rate $\eta$, bootstrap sample sizes $(\ell_n)_{1\le n\le N}$, initial distribution $q_0$, kernel $K$.}
\KwOut{Weighted particles $(X_n,W_{n,N})_{1 \le n\le N}$ }
\nl Generate $X_1$ from $q_{0}$ and set $W_{1,1} \leftarrow \p{\frac{f_u(X_1)}{q_{0}(X_1)}}^\eta \gamma_1$\;

\nl\For{$n\leftarrow 1$ \KwTo $N-1$}
{\nl generate independent indices $(u\p{i,n})_{1\le i\le l_{n}}$ from $\proba_{n} = \frac{\sum_{i=1}^{n} W_{i,n}\delta_{i}}{ \sum_{i=1}^{{n}} W_{i,n}}$ \; 
\nl generate $X_{n+1}$ from the mixture $q^*_{n}$ defined in \eqref{eq:q*}\;
\nl $W_{n+1,n+1}\leftarrow \p{\frac{f_u(X_{n+1})}{q^*_{n}(X_{n+1})}}^\eta \gamma_{n+1}$\;
\nl for all $1\le i\le n$, $W_{i,n+1} \leftarrow (1-\gamma_{n+1}) W_{i,{n}}$\;
}
\end{algorithm}

The computing cost of the $n$-th iteration is then of order  $\ell_n \log(n)$ instead of $\ell_n \log(n)$ in Algorithm \ref{alg1},
which, when neglecting the operation to update the past weights (line $6$ in Algorithm \ref{alg2} compared to line $5$), leads to an overall computing time $ c_K  \sum_{n = 1} ^N  \ell_n \log(n) + Nc_f$. The assumption that line $6$ is negligible compared to line $5$ is observed in practice when using the Gaussian kernel. More importantly, in the case when $\gamma_i = 1/i$, line $6 $ is not necessary anymore and therefore leading to the mentioned computing time.

%Then, the cost of generating the sample $(X^*_{i,n})_{1\le n\le l_{N}}$ according to $\proba_{n}$ is of order $\ell_n \log(n)$ ($\ell_n$ multinomial draws). In practice, we choose $\ell_n \propto n^\delta$ with  $\delta \in (0,1)$ so $\ell_n\log(n)$ is negligible compare to $n$.
%The cost of evaluating $q_n^*$ is of order $C l_k$.
%Thus, the cost of one step $n\ge 1$ of Algorithm 2 is of order $C \ell_n + c_f + n$. 

\subsection{Mini-batching variant}\label{sec:minibatch}
% In practice, as often seen in AIS, we use mini batch $m$. At each step $n+1$, it consist in generating $m$ particles $X_{n+1,1}\dots X_{n+1,m}$ of density $q_{n}$ instead of one. And the unnormalized density update $\eqref{eq:g}$ is replaced by:
% \[
%     g_{n+1}(x = (1-\gamma_{n+1})g_n + \gamma_{n+1}\frac 1m \sum_{k=1}^m w^\eta_{n+1,k} K_{b_{n+1}}(x-X_{n+1,k})
% \]
% for all $x\in \reels^d$, where $w_{n+1,k}= \tfrac{f}{q_n}(X_{n+1,k})$.

% This technique pose the advantage of reducing the number of updates while keeping the same number of particles. Therefore the computational time is lower. However, we also reduce the number of updates. 
% Thus, the choice of $m$ is a trade-off between adaptation and computational time.

% In our experiments we choose $m=300$.
In the context of stochastic algorithms, the utilization of mini-batches, each consisting of \(m\) particles, is a common methodological refinement.  At each iteration \(n+1\), instead of generating a single particle \(X_{n+1}\) according to the density \(q_n\), the algorithm generates \(m\) particles \(X_{n+1,1}, \ldots, X_{n+1,m}\), each sampled from the same density \(q_n\).

The update formula for the unnormalized density, originally delineated by \eqref{eq:g}, is accordingly modified to accommodate this batch processing strategy:
\begin{equation*}
    g_{n+1}(x) = (1-\gamma_{n+1})g_n(x) + \gamma_{n+1} \frac{1}{m} \sum_{k=1}^{m} w^\eta_{n+1,k} K_{b_{n+1}}(x - X_{n+1,k})\, , \qquad \forall x \in \mathbb{R}^d,
\end{equation*}
 where the weights \(w_{n+1,k}\) are computed as \(w_{n+1,k} = {f(X_{n+1,k})} / {q_n(X_{n+1,k})}\).

This mini-batch approach offers the advantage of preserving the total number of particles, while reducing the computational time as it can generate multiple particles at each step through parallelization. However, it is worth noting that the fewer number of updates potentially impacts the adaptivity of the sampling mechanism. Thus, the choice of \(m\) represents a trade-off between computational efficiency and adaptive capability.

It is noteworthy that the convergence properties of our algorithm remain intact even under mini-batch adaptations. Indeed, the conditional expectation \(\mathbb{E}[g_{n+1}|\mathcal{F}_n]\) remains invariant regardless of whether mini-batching is employed. Consequently, we can extend the proof to the mini-batched context.
Moreover, as \(m\) is a predetermined constant in our mini-batching setup, the martingale increment term \(\mathbb{E}[g_{n+1}|\mathcal{F}_n] - g_{n+1}\) retains its asymptotic characteristics. This further substantiates that the introduction of mini-batching does not perturb the  convergence behavior of the algorithm.

\section{Numerical experiments}\label{s5}

This section is dedicated to the practical evaluation of the MIDAS algorithm (in particular the subsampling version) based on several synthetic examples as well as a real data Bayesian estimation problem. Two recent competitors, from Markov chain Monte Carlo and Sequential Monte Carlo literature, shall be considered for the sake of comparison.

\subsection{Competitors}\label{sec:setup}

We use the subMIDAS algorithm as described in Algorithm ~\ref{alg2}, with $\ell_n=\sqrt n$ and mini-batches of size $m=300$ as described in Section \ref{sec:minibatch}. The bandwidths, mixture weights and step sizes are given by 
$$b_n = \frac{0.4}{\sqrt{d}} (\frac{m n}{10000} + 1)^{-1/(4+d)}\, , \quad \lambda_n = \frac{1}{ \log(m n+10)} \,,\quad \gamma_n =\frac{ 1 } { (n+10)}. $$
To allow reasonable initialization, we implement a burn-in phase as follows: at \( n = 1 \), we set an initial batch size of \( m_0 = 2000 \) and for the first ten steps (\( n \le 10 \)), we set \( \lambda_n = 0.5 \).

A key feature of our algorithm is that it generates random variables with a density known up to a normalization constant, and it does not require the gradient of the density. We hence compare with two other algorithms sharing the same specifications: Annealed Importance Sampling (AIS)~\cite{neal2001annealed} and Kernel Adaptive Metropolis-Hastings (KAMH)~\cite{sejdinovic2014kernel}.

For AIS, we use a batch size of 300 with 20 Metropolis updates. Therefore, for a given number \( K \) of intermediate distributions, we evaluate the unnormalized target density function \( K \times 300 \times 20 \) times.  
We use intermediate distributions in the form described by~\cite{neal2001annealed}, with a geometrically spaced schedule. We run AIS independently for various numbers of intermediate distributions \( K \), and in the final iteration, we plot the distance to the target distribution with respect to the number of evaluations of the target, \( K \times 300 \times 20 \).

For KAMH, in each run, the starting particle is generated from \( q_0 \). We use a Gaussian kernel with a covariance matrix \(\sigma^2 I_d\) where \(\sigma = 5\), and we set the scaling parameters to \(\nu = {2.38}/{\sqrt{d}}\) and \(\gamma = 0.2\). In this case, one step corresponds to one evaluation of the unnormalized target density function. 

 \begin{figure}[!ht]
     \centering
     \includegraphics[width=10cm]{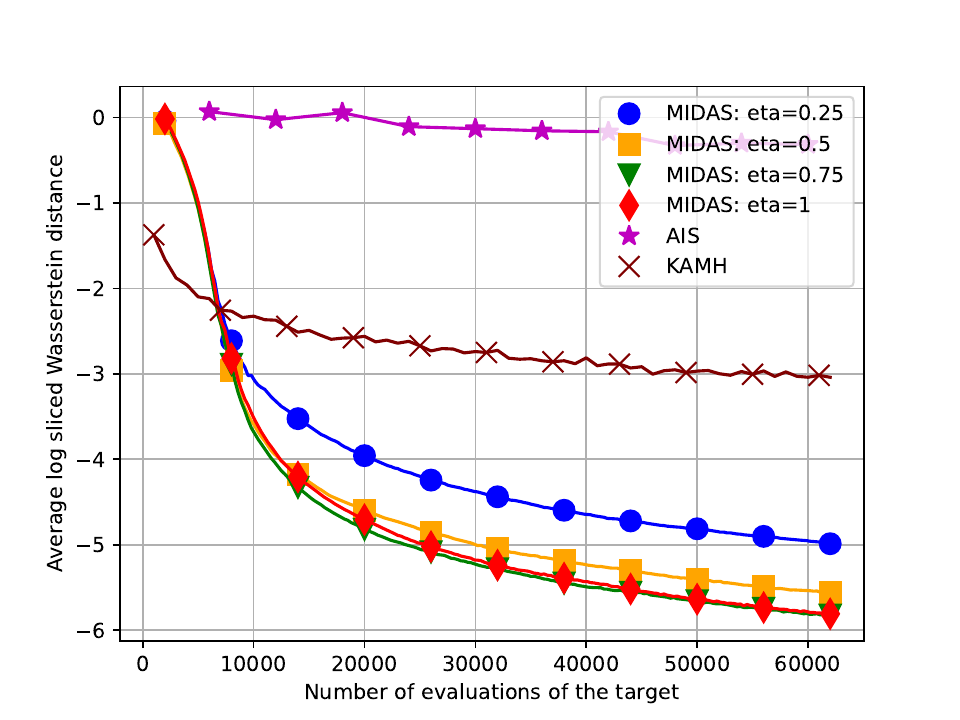}
     \caption{Cold start}
     \label{fig:coldstart}
 \end{figure}

\subsection{Evaluation}
For each competing method, we evaluate the distance to the target distribution with respect to the budget, i.e., the number of evaluation of $f_u$.
After a given number of evaluations of $f_u$, each method (subMIDAS, AIS, KAMH) gives an empirical distribution described with the help weighted particles. 
To evaluate the performance, we compute the sliced Wasserstein distance between the empirical distribution and the target distribution. Recall that for two probability measures \(\mu\) and \(\nu\) on \(\mathbb{R}^d\) with finite second moments,
\[
    SW_2(\mu,\nu) :=\bE_{\theta\sim \mathcal U(\mathbb{S}^{d-1})} [W_2(\theta_\#\mu, \theta_\#\nu)^2 ]\,.
\]
Here, \( W_2 \) denotes the Wasserstein distance on probability measures on \(\mathbb{R}\) with finite second moments, \(\mathbb{S}^{d-1}\) denotes the unit sphere \(\{\theta \in \mathbb{R}^d : \|\theta\| = 1\}\), $\mathcal U(\mathbb S^{d-1})$ denotes the uniform distribution on $\mathbb S^{d-1}$, and \(\theta_\#\mu\) denotes the pushforward of the measure \(\mu\) through the map \( x \in \mathbb{R}^d \mapsto \langle \theta, x \rangle \). 
With the help of Monte Carlo simulation from $\mathcal U(\mathbb S^{d-1})$, this distance is easily approximated since the Wasserstein distance \( W_2 \) on measures on \(\mathbb{R}\) admits a closed form.
In the presented graphs, we average the sliced Wasserstein distance over 50 runs.

\subsection{Toy examples}

% Our results are presented in Figures~\ref{fig:toyM} and  a smaller learning rate yields much better results, as expected.  We averaged the results over 50 runs of the algorithm, so Figures~\ref{fig:toyM} shows that for all learning rates, there are cases where there is no convergence; however, there are fewer such cases when the learning rate is smaller. We also note that convergence is slower when the learning rate is low.

\subsubsection*{Cold start}

The "cold start" scenario occurs when the target density is far from the initial density.
In this case, our target distribution is a Gaussian with a mean of \( {5}/{\sqrt{d}} 1_d \) and a variance of \( {{(0.4)^2}}/{d} I_d \). The initial distribution is a Gaussian with a mean of \( 0 \) and a covariance matix: \( {5}/{d} I_d \).

The results are given in Figure \ref{fig:coldstart}. In such a simple example of a target with a single mode, $\eta=1$ seems to be a reasonable choice. While taking $\eta=3/4$ makes the algorithm slightly faster at the start, it appears that taking a small value for $\eta$ might slow the convergence of MIDAS. AIS performs poorly, as the nature of the algorithm leads to high variance in the weights when the starting distribution is far from the target.

 \subsubsection*{Gaussian mixture}
In this case, the target distribution is a  mixture of two Gaussians with equal weights and  mean equal to \( {1}/({2\sqrt{d}}) 1_d \) and \( -{1}/({2\sqrt{d}}) 1_d \). The covariance matrix of the two Gaussianns is \( {0.4^2}/{d} I_d \). The initial distribution is a Student's distribution with mean \( 0 \) and scale parameter \( {5}/{d} I_d \).

In this  multi-modal target example, we clearly see the importance of \(\eta\). In this case, the algorithm performs better for the lowest value of \(\eta\). Furthermore, for all four values of \(\eta\), MIDAS outperforms our competitors.
  \begin{figure}[!ht]
     \centering
     \includegraphics[width=10cm]{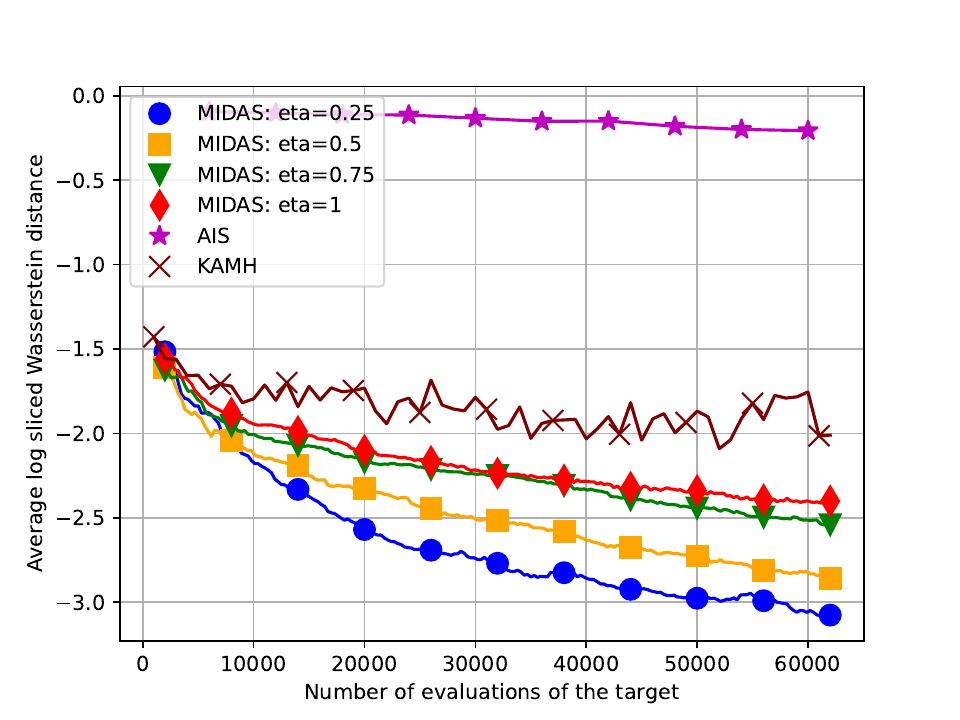}
     \caption{Gaussian Mixture}
     \label{fig:toyM}
 \end{figure}

 \subsection*{Anisotropic Gaussian mixture}
 In this case, the target distribution is a  mixture of two Gaussians with  mean equal to \( {1}/({2\sqrt{d}}) 1_d \) and \( -{1}/({2\sqrt{d}}) 1_d \). The covariance matrix of the two Gaussianns is \( 0.4^2/d \text{Diag}(10,1,\dots,1) \). The initial distribution is a Student's distribution with a location parameter of \( 0 \) and a scale parameter of \( {5}/{d} I_d \).

This case produces results that are similar to those of the Gaussian mixture case. 
  \begin{figure}[!ht]
     \centering
     \includegraphics[width=10cm]{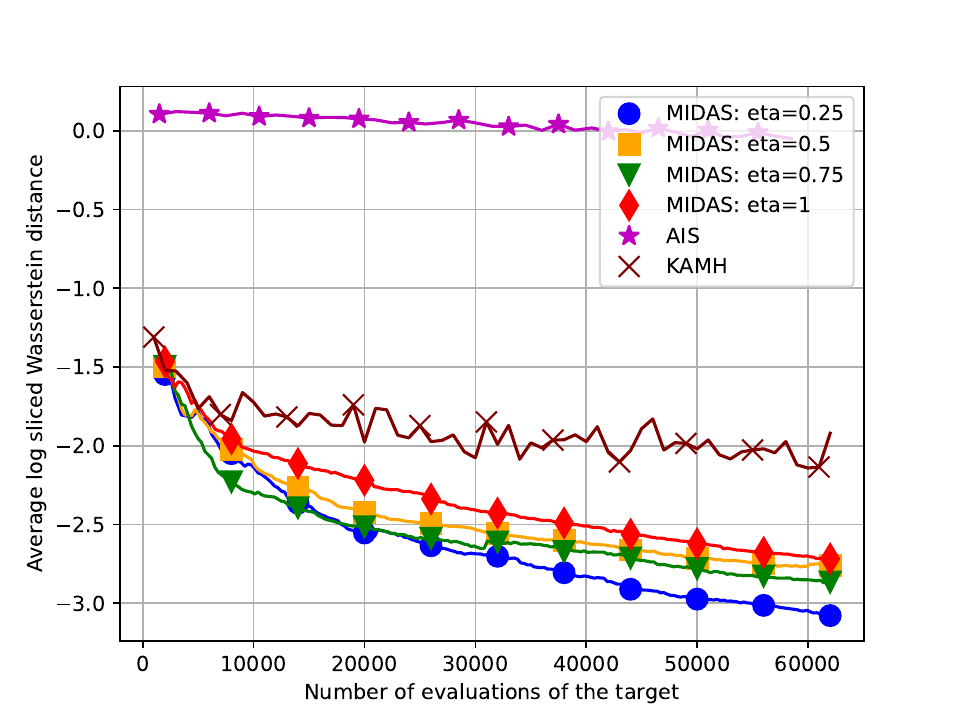}
     \caption{Anisotropic Gaussian Mixture}
     \label{fig:toyA}
 \end{figure}
 
\subsection{Bayesian logistic regression in real data set}

We consider the Bayesian logistic regression setting of \cite{gershman2012nonparametric}, also considered in the recent Bayesian inference literature \citep{liu2016stein,korba2022adaptive,li2022sampling}. More precisely, we consider a data set $\mathcal D = \{c_i,z_i\}_{i\in I}$ of points $z_i\in\reels^{d-1}$ and  and binary class labels $c_i\in\{-1,1\}$ for $i\in I$ where $d$ is the dimension.

The model is parameterized by a parameter $\theta = (w,\beta)\in \reels^{d-1} \times \reels$ and the hyper parameters $a,b \in \reels$.  For each $i\in I$, we have
\begin{equation}\label{eq:bayR}
        \proba\p{c_i = 1\vert \theta, z_i } = \frac 1{1+\ex^{-w^Tz_i}}\,.
\end{equation}
The parameter $\beta$ follows a Gamma distribution of shape parameter $a=1$ and rate parameter $b=0.01$. The parameter $w$ conditionally to $\beta$ follows a Gaussian distribution of mean $0$ and variance $ (1 / \beta)  I_{d-1}$.

We are interested in estimating the posterior density $p(\theta\vert \mathcal D_\text{train})$ of the parameter $\theta$ according to a training data set $\mathcal D_{\text{train}}$ which is given by:
\[
    p(\theta \vert \mathcal D_{\text{train}}) =\frac{ p(\mathcal D_{\text{train}}\vert \theta) p(\theta)}{\proba(\mathcal D_{\text{train}})}
\]
where the prior $p(\theta)$ is the density of the parameter $\theta$ given by the model, the likelihood $p(\mathcal D_\text{train}\vert \theta)$ is the probability of obtaining the the data set $\mathcal{D}_\text{train}$ with parameter $\theta$ and is given by~\eqref{eq:bayR} and the  marginal likelihood $\proba(\mathcal D_\text{train})$ is the probability of obtaining the data set $\mathcal D_\text{train}$ which does not need to be computed here. Thus, the unnormalized target distribution is given by $p(\mathcal D_{\text{train}}\vert \theta) p(\theta)$.

Given a new data point $z_{\text{new}}$, we are interested in predicting the label $c_{\text{new}}$. Using the posterior density and~\eqref{eq:bayR}, we have: $ \proba(c_\text{new}=1|z_\text{new},\mathcal D_{\text{train}}) = \int \proba(c_\text{new}=1\vert \theta, z_\text{new})p(\theta\vert \mathcal D_{\text{train}})\dr\theta$. And if the computed probability exceeds $1/2$, we infer that $c_{\text{new}} =1$. Otherwise, we assign  $c_{\text{new}} =-1$.

%and the accuracy of the algorithm is given by the expectation of obtaining $\mathcal D_{\text{test}}$ when the parameter $\theta$ has a density $p(\theta\vert \mathcal D_\text{train})$. Therefore, the accuracy is $\int p(\mathcal D_\text{test}\vert \theta )p(\theta \vert \mathcal D_{\text{train}})\dr \theta$ which our algorithm can estimate.

  We consider the dataset 'waveform' made of $5000$ entries with dimension $d=22$. For each competing method, the predictions are made based on the training set  $\mathcal D_{\text{train}}$ of size $400$ and the  average accuracy is computed with the help of the test dataset made of the remaining $4600$ points. This is displayed in Figure~\ref{fig:bayReg} where we observe that MIDAS with $\eta=1/4$ outperforms all competitors. Moreover, in this case, the classical algorithm with $\eta=1$ performs even worse than the competitors like AIS and KAMH.
 \begin{figure}
     \centering
     \includegraphics[width=10cm]{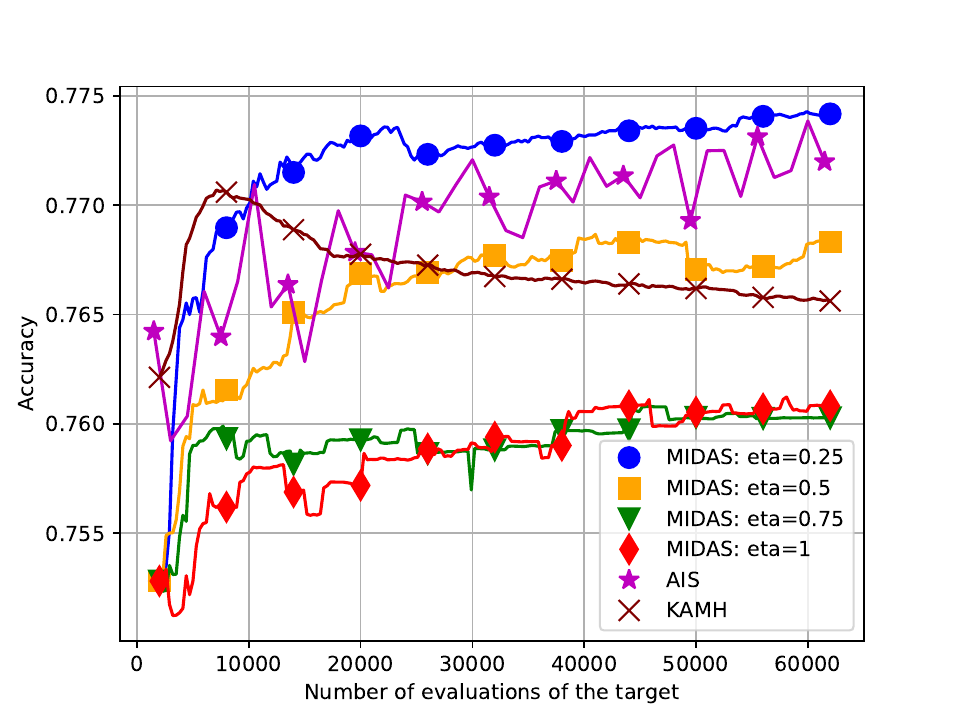}
     \caption{Bayesian Logistic Regression}
     \label{fig:bayReg}
 \end{figure}

\section{Possible extensions}

Several extensions of this work may be worth exploring further. We conjecture that the subsampling algorithm subMIDAS satisfies similar convergence properties as the one described in Theorem \ref{th:main} depending on the choice of the parameters $\lambda_n$, $\gamma_n$ and $b_n$ with respect to $\ell_n$. This is left for further research. In addition, we believe that the assumptions on the previous sequences, in case $\eta < 1/2 $, might be alleviated at the price of a more thorough examination.

The adaptive choice of the learning rate $\eta$ might help to improve the convergence rate and the practical behavior of the algorithm. As noted in~\cite{chopinconnection}, our algorithm can be viewed as an annealed importance sampling (AIS) algorithm. Such an algorithm chooses a distribution path $f_k$ of the form $f_k \propto q_0^{1-\beta_k}f^{\beta_k}$, where the annealing schedule satisfies $\beta_N=1$ for the final iteration $N$. The annealing schedule corresponding to our algorithm is $\beta_n = 1 - (1 - \eta)^n$ as studied in~\cite{chopinconnection}, for instance. There are many choices of schedules that work well for different types of problems. The existence of an optimal schedule that minimizes the variance of the estimates is discussed in~\cite{syed2024optimised}. Studying optimal choice of the parameter $\eta$ in our context could have some interesting connections with the previously mentioned papers. We leave this investigation for future work.

\bibliographystyle{alpha}
\bibliography{bibli}

\newpage

\begin{appendix}

\section{Proofs}\label{s:app}

\subsection{Notation}

Let $\cC(\bR^d)$ denote the set of continuous functions on $\bR^d$.
We denote by $B(x,\varepsilon)$ the $d$-dimensional open ball of radius $\varepsilon$ and center $x$.
%We denote by $\overline A$ the closure of a set $A$.
We define $(x)^+ = \max(x,0)$. 
We use the convention that $\inf\emptyset = +\infty$, and that $0^0=1$.
For positive sequences $(a_n)_{n\ge 1}$, $(b_n)_{n\ge 1}$, the notation $a_n = O(b_n)$ means that there exists $C>0$ such that
$a_n\le C b_n$ for all $n\ge 1$. For a sequence $(a_n)_{n\ge 0}$, we denote $\udl \lim a_n := \lim_{n\to\infty}\inf_{k\ge n} a_k$ and $\ovl \lim a_n := \lim_{n\to\infty}\sup_{k\ge n} a_k$ .

\subsection{Proof of Theorem \ref{th:main}}\label{sec:sketch}
%The proof is made of several steps. Some of them will require auxiliary results that will be proved in a dedicated section. 

%\textbf{Step 0: Normalization.}
We remark that the sequence $(q_n)_{n\ge 0}$ constructed with $f_u$ is the same as the one using $f$. Indeed, for a target $f$ normalized, the algorithm gives two sequences $(q_n)_{n\ge 0}$ and $(g_n)_{n\ge 0}$. Now we define $(\tilde q_n)_{n\ge 0},(g_n)_{n\ge 0}$ the sequences associated to the algorithm with target $\tilde f =C  f$. We see by induction that for all $n\ge 0$, $\tilde g_n = C^\eta g_n$ and $\tilde q_n = q _n$.
Then, in the rest of the proof we will consider $f_u=f$ so that for each $n \ge 1$,
$$ w_n = \frac{f(X_n)} { q_{n-1} (X_n)}.$$
We define the operator $T$ on the set of non-negative continuous functions on $\bR^d$ by:
$$
Tg : x\mapsto f(x)^\eta g(x)^{1-\eta}\,,
$$
for all $g$.
The iterates $(g_n)$ given by~(\ref{eq:g}) can be rewritten as:
\begin{align}
g_{n+1} = (1-\gamma_{n+1})g_n + \gamma_{n+1} Tq_n\ast K_{b_{n+1}} + \gamma_{n+1} \xi_{n+1}\,,\label{eq:gn}
\end{align}
where $\xi_{n+1}$ is a r.v. on $\cC(\bR^d)$ given by:
\begin{align}
\xi_{n+1}(x):= w_{n+1}^\eta K_{b_{n+1}}(x-X_{n+1})-  Tq_n\ast K_{b_{n+1}} (x) \,.\label{eq:xi}
\end{align}
We remark that, for every $x\in\bR^d$, $\bE(\xi_{n+1}(x)|\cF_n)=0$.
Iterating~\eqref{eq:gn} with
\begin{equation}\label{eq:psi}
  \psi_{i+1}^n := \prod_{j=i+1}^n (1-\gamma_j)\, ,
\end{equation}
we obtain for every $n\ge 0$
\begin{equation}\label{eq:gsum}
    g_n  =\sum_{i=1}^n \psi_{i+1}^n\gamma_i Tq_{i-1}\ast K_{b_{i}}+ M_n  \,,
\end{equation}
where the martingale term is defined as
\begin{equation}\label{eq:M}
  M_n := \sum_{i=1}^n \psi_{i+1}^n\gamma_i\xi_i \, .
\end{equation}
We define $Z_n := \int g_n$.
  By integrating~(\ref{eq:gn}), we have:
  \[
    Z_{n+1} = (1-\gamma_{n+1})Z_n + \gamma_{n+1}\int Tq_n +\gamma_{n+1}\int \xi_{n+1}.
  \]
We are now ready to proceed with the main steps of the proof.

\textbf{Step 1: Control of the martingale term.}
We obtain the following bound on the martingale term $M_n$ defined previously.
\begin{prop}\label{prop:mtgl} 
 % Suppose $ f=f_u$,
  Let Assumptions~\ref{hyp:fq}, \ref{hyp:K} and~\ref{hyp:rate} hold true.
  There exists $\varepsilon>0$ such that for any $p>0$
  the process $M_n$ of equation~(\ref{eq:M}) satisfies: % for every $\omega\in \Omega$,
  % \begin{align}\label{supM}
  % &\sup_{\|x\|\le n^p}| M_n(x)| \le C_p n^{-\varepsilon-(1-\eta)\kappa}
  % \end{align}
    \begin{align}\label{supM}
  & \lim_{n\to\infty} \frac{\sup_{\|x\|\le n^p}| M_n(x)|}{\lambda_n^{1-\eta}} =0\,,
  \end{align}
  almost surely.
  \end{prop}

\textbf{Step 2: An upper bound on $Z_n$.}  By applying Hölder's inequality, $\int Tq_n \le 1$.
  Consequently, if we define the sequence $(\ovl Z_n)_{n\ge 0}$ with $\ovl Z_0 =Z_0$ as
  \begin{equation}\label{eq:Z_high}
    \ovl Z_{n+1}  := (1-\gamma_{n+1})\ovl Z_n + \gamma_{n+1} +\gamma_{n+1}\int \xi_{n+1}\, ,
  \end{equation}
  we obtain $Z_n\le\ovl Z_n$ for every $n$ a.s. .

 We obtain the following result.
\begin{prop}\label{prop:minZ} Suppose $f=f_u$, let Assumptions~\ref{hyp:fq} and \ref{hyp:rate} hold true. $(\ovl Z_n)_{n\ge 0}$ defined in~\eqref{eq:Z_high} satisfies $ \lim_{n\to\infty}\ovl Z_n =1$, a.s..
\end{prop}

\textbf{Step 3: $g_n$ is away from $0$.}
By applying a convexity inequality ($1-\eta  < 1$) we obtain:
\[
Tq_n(x) \ge \p{\frac {f(x)} 2}^ \eta\p{(1-\lambda_{n})^{1-\eta}\p{\frac{g_n(x)}{Z_n}}^{1-\eta} +\lambda_n^{1-\eta}q_0(x)^{1-\eta} }\, ,
\]
for every $x$.
We define the operator $\udl T$ on the set on non-negative continuous function on $\bR^d$ as
\[
        \udl Tg \,:\, x\mapsto (\tfrac{f(x)}{2})^\eta g(x)^{1-\eta}\,,
\]
for every $g$.
Then, using that $\ovl Z_n \ge Z_n$, we obtain
\[
g_{n+1}\ge (1-\gamma_{n+1})g_n + \gamma_{n+1}\frac{(1-\lambda_n)^{1-\eta}}{ \ovl Z_n^{1-\eta}}\udl Tg_n\ast K_{b_{n+1}} +\gamma_{n+1}\lambda_n^{1-\eta}\udl Tq_0\ast K_{b_{n+1}} + \gamma_{n+1}\xi_{n+1}\, ,
\]
for every $n\in\bN$.
Iterating the latter equation, we obtain for every $n\ge 0$
\[
    g_n \ge v_n + M_n +  \sum_{i=1}^n\psi_{i+1}^n \gamma_i \lambda_i^{1-\eta}   \udl T q_0\ast K_{b_i} \,,
\]
where 
\[
    v_n := \sum_{i=1}^n\psi_{i+1}^n \gamma_i    \frac{(1-\lambda_i)^{1-\eta}}{\ovl Z_n^{1-\eta}}\udl T g_{i-1}\ast K_{b_i}\,.
\]

By the martingale control given by Proposition~\ref{prop:mtgl}, we obtain:
\begin{lemma}\label{lem:gsupv} Let Assumptions~\ref{hyp:fq},~\ref{hyp:K} and~\ref{hyp:rate} hold true. For any $\alpha>0$ and $p>0$, we have that almost surely, there exists $ N_{\alpha,p}\in (0,\infty)$ such that for every $n\ge N_{\alpha,p} $
\begin{equation}\label{eq:gnsupvn}
  \min_{f(x)\ge \alpha, \norm{x}\le n^p}\p{ g_n -v_n}(x) \ge 0\, .
\end{equation}
\end{lemma}

The previous result allows to show that $g_n$ cannot reach $0$ in places where $f$ is positive.
%Eq.~\eqref{eq:gnsupvn} will be a crucial point for avoiding the trap $g_n=0$.

\begin{prop}\label{prop:minorationv}
        Let Assumptions~\ref{hyp:fq},~\ref{hyp:K} and~\ref{hyp:rate} hold true.  
Then for every $\varepsilon>0$, and every $x\in\bR^d$ satisfying $\inf_{y\in B(x,\varepsilon)}  f(y)>0$, we obtain  
\begin{gather*}
  \varliminf \inf _{y\in B(x,\varepsilon)}{v_{ n}(y) }> 0\,,\\
    \varliminf \inf _{y\in B(x,\varepsilon)}{g_{ n}(y) }> 0\,,
\end{gather*} 
almost surely.
\end{prop}

\textbf{Step 4: A lower bound on $g_n$}
%Thanks to the latter step, we obtain:
\begin{prop}\label{prop:minorationgparf}
       Let Assumptions~\ref{hyp:fq},~\ref{hyp:K} and~\ref{hyp:rate} hold true.  
Then for every compact set $A\subset \bR^d$, almost surely, we obtain  
\[
  \varliminf \inf _{x\in A }{g_{ n}(x)- f(x) }\ge  0\,.
\] 
\end{prop}
 According to Proposition~\ref{prop:minorationgparf}, we obtain
  \begin{equation}\label{eq:prop3}
    \varliminf \inf_{x\in A}\p{g_n(x)-f(x)}\ge 0 \,,
  \end{equation}
  for every compact set $A\subset \reels^d$.
  Let us fix $\varepsilon>0$, by Assumption~\ref{hyp:fq}, there exists a compact set $A_\varepsilon$ such that:
  \[
      \int_{A_\varepsilon} f(x)\dr x \ge 1-\varepsilon\, .
  \]
Moreover,
  \[
    \int_{A_\varepsilon}\p{g_n -f}\ge \abs{A_\varepsilon}\inf_{x\in A_\varepsilon}\p{g_n(x)-f(x)} \,. 
  \]
  By~(\ref{eq:prop3}), we obtain
  \[
    \varliminf Z_n \ge \varliminf \int_{A_\varepsilon}g_n\ge\int_{A_\varepsilon} f \ge 1-\varepsilon\,,
  \]
  for every $\varepsilon>0$, which gives the result.

\begin{equation}\label{eq:Zn}
  Z_n\xrightarrow{a.s.} 1  \,,
\end{equation}
under the assumption of Proposition~\ref{prop:minorationgparf}.

\textbf{Step 5: An upper bound on $g_n$}
\begin{prop}\label{prop:majoration}
  
Suppose $f=f_u$, let Assumptions~\ref{hyp:fq},~\ref{hyp:K},~\ref{hyp:rate} hold true. 
Then, for every compact $A\subset \reels^d$, 
\[
  \varlimsup \sup _{x\in A}\p{ g_{n}(x) - f(x)}\le 0\,.
\] 
\end{prop}
And, using the Propositions~\ref{prop:minorationgparf} and~\ref{prop:majoration}, we obtain Theorem~\ref{th:main}, leading to Corollary~\ref{coro:TCL}.

  % And for $(Z_n)$, consider $g$ associated to $f$, then,  $\int g_n \to 1$ almost surely. By Step 0, $g_n = \p{\int f_u}^\eta g_n $ which conclude the proof.

% \subsection{Auxiliary results}

\subsection{Proof of Proposition ~\ref{prop:mtgl} }
In this subsection, we suppose $f=f_u$, and we let Assumptions~\ref{hyp:fq},~\ref{hyp:K}, and~\ref{hyp:rate} hold true. 
The martingale increment $(\xi_i)_{i\ge 1}$ can also be rewritten as:
\begin{equation}\label{eq:newxi}
  \xi_i(x) = w_i^\eta K_{b_i}(x- X_i) - \espcond{w_i^\eta K_{b_i}(x- X_i)}{\cF_{i-1}},
\end{equation}
where $X_i \sim q_{i-1}$ and $w_i = \frac{f}{q_{i-1}}(X_i)$. 
By Hölder's inequality, we obtain a useful inequality
\begin{equation}\label{eq:holder}
\espcond{w_{n+1}^\eta}{\cF_n}\le 1\, .
\end{equation}

The purpose of the next  lemmas is to prove Proposition~\ref{prop:mtgl}.
\begin{lemma}
  \label{lem:bounds-Y}
  There exists a constant $C>0$, depending only on $K$ and $f$, such that
  the following statements hold for all $n\ge i\ge 0$, $x,y\in \bR^d$, 
  \begin{align}
  &|\xi_i(x+y) - \xi_i(x)|
  \le C \lambda_{i-1}^{-\eta} b_i^{-d-1} \|y\|\\
  &|\xi_i(x)| \le   C \lambda_{i-1}^{-\eta} b_i^{-d} \\
  &\bE( \xi_i(x)^2|\cF_{j-1}) \le C {b_i^{-d\p{1+(1-2\eta)_+}} \lambda_{i-1}^{-(2\eta-1)_+}}\, .
  \label{xideux}
  \end{align}
  \end{lemma}
  \begin{proof}
    Notice first that the following bound is a consequence of Assumption~\ref{hyp:fq}-\ref{hyp:fq:f<q0} and~(\ref{eq:q}):
    \begin{align}\label{mbound}
    w_i^\eta \le   \p{c\lambda_{i-1}}^{-\eta}.
    \end{align}
    We establish the first point. 
    Set $\Delta_i(x,y) := K_{b_i}(x+y-X_{i})-K_{b_i}(x-X_i)$ and remark that 
    $|\Delta_i(x,y)|\le L_K\|y\|b_i^{-d-1}$, where $L_K$ the Lipschitz constant of $K$ (cf. Assumption~\ref{hyp:K}-\ref{hyp:K:reg}). 
    Recalling the definition of $\xi_i$ in~(\ref{eq:newxi}), we obtain:
    \begin{align*}
    |\xi_i(x+y) - \xi_i(x)| 
    &= \abs{w_i^\eta\Delta_i(x,y)-\bE(w_i^\eta \Delta_i(x,y)|\cF_{i-1})} \\
    &\le w_i^\eta|\Delta_i(x,y)|+\bE(w_i^\eta |\Delta_i(x,y)|\,|\cF_{i-1}) \\
    &= (w_i^\eta+\bE(w_i^\eta |\cF_{i-1}))L_K\|y\|b_i^{-d-1}\,.
    \end{align*}
    We conclude by using~(\ref{eq:holder}) and (\ref{mbound}).
    The second point follows from:
    \begin{align*}
    |\xi_i(x)| 
    &= | w_{i}^\eta K_{b_i}(x-X_i)- \bE(w_{i}^\eta K_{b_i}(x-X_i)|\cF_{i-1})|\\
    &\le \|K\|_\infty  b_i^{-d}( w_{i}^\eta + \bE(w_{i}^\eta |\cF_{i-1}))\,,
    \end{align*}
    and~(\ref{eq:holder}) and (\ref{mbound}) again.
    We establish the third point. 
    \begin{align*}
    \bE\big[\xi_{i+1}(x)^2|\cF_i\big]
    &\le \bE\big[ w_{i+1}^{2\eta}K_{b_{i+1}}(x-X_{i+1})^2|\cF_i\big]\\
    &=\int f(y)^{2\eta}q_i(y)^{-2\eta}K_{b_{i+1}}(x-y)^2q_i(y)dy\\
    &\le \|f\|_\infty^{2\eta}\int K_{b_{i+1}}(x-y)^2q_i(y)^{1-2\eta}dy.
    \end{align*}
    If $\eta<\frac12$, we use the Hölder inequality:
    \begin{align*}
    \bE\big[\xi_{n+1}(x)^2|\cF_n\big]
    &\le \|f\|_\infty^{2\eta}\Big(\int K_{b_{i+1}}(x-y)^{1/\eta}dy\Big)^{2\eta}\Big(\int q_i(y)dy\Big)^{1-2\eta}\\
    &\le \|f\|_\infty^{2\eta}\Big(\int b_{i+1}^{-d/\eta}K\big(b_{i+1}^{-1}(x-y)\big)^{1/\eta}dy\Big)^{2\eta}\\
    & =\|f\|_\infty^{2\eta} b_{i+1}^{-2d(1-\eta)}\|K\|_{1/\eta}^2,
    \end{align*}
    which is finite by Assumptions~\ref{hyp:fq}-~\ref{hyp:fq:reg} and~\ref{hyp:K}-~\ref{hyp:K:reg}. 
    If $\eta\ge \frac12$, we use that $  q_{i} \ge  \lambda_i q_0 $ (cf.~(\ref{eq:q})) 
    and Assumption~\ref{hyp:fq}-\ref{hyp:fq:reg}:
    \begin{align*}
    \bE\big[\xi_{i+1}(x)^2|\cF_i\big]
    &\le\|f\|_\infty^{2\eta}\lambda_i^{1-2\eta}\int K_{b_{n+1}}(x-y)^2q_0(y)^{1-2\eta}dy\\
    &\le \|f\|_\infty^{2\eta} \|q_0\|_\infty^{1-2\eta}\norm{K}^2_2 \lambda_i^{1-2\eta}b_{i+1}^{-d} 
    \end{align*}
    In any case,~(\ref{xideux}) is satisfied.
  \end{proof}
   \begin{lemma}
  %Suppose $f =f_u$, let Assumptions~\ref{hyp:fq} and~\ref{hyp:K} hold true. 
  There exist a constant $C$ and a rank $n_0$ depending only on quantities given in the assumptions, such that, 
  for all $A>0$, $n\ge 1$ and $q\ge 1$:
  \begin{align*}
  &\bP\left(\sup_{\|x\|\le A}|M_n(x)| > Cq \tau_n\sqrt{\log n} \right) 
  \le C\p{1+ A \frac{r_n}{\tau_n}}^dn^{-{q}}~~~~n\ge n_0\\
  &\tau_n:=\p{n\lambda_n^{-(2\eta-1)_+} b_n^{-d(1+(1-2\eta)_+)}}^{1/2}\\
&r_n=n\lambda_n^{-\eta}b_n^{-d -1} \, .
  \end{align*}
  \label{prop:freedman-appli}
  \end{lemma}
  \begin{proof}
  Let $n\in \bN^*$.
  For any $i\le n$, and any $x\in\bR^d$, we define
  \[
  \zeta_{i,n}(x) := \gamma_{i}\psi_{i+1}^n\xi_i(x)\,.
  \]
  Note that $\bE(\zeta_{i,n}(x)|\cF_{i-1}) = 0$ for all $i\le n$.
  We will apply Proposition~\ref{prop:freedmanx}. We consider $\varepsilon>0$ to be chosen later, and,
  following this proposition, denote
  \begin{align*}
  \tilde \zeta_{i,n}(x) := \sup_{\|y\|\le \varepsilon}|\zeta_{i,n}(x+y)-\zeta_{i,n}(x)|\, .
  \end{align*}
  By Lemmas~\ref{lem:bounds-Y},~\ref{lem:psigamma} and Assumption~\ref{hyp:rate}, the following inequalities hold,
  for all $x\in \bR^d$:
  \begin{align*}
  &  \max_{i=1,\dots,n} |\zeta_{i,n}(x)| 
    \le  C\gamma_n  \lambda_{n-1}^{-\eta} b_n^{-d}
    =:m\\
  & \sum_{i=1}^n\bE(\zeta_{i,n}(x)^2|\cF_{i-1}) 
    \le C \gamma_n ^2 \sum_{i=1}^n\p{ b_i^{-d-d(1-2\eta)_+} \lambda_{i-1}^{-(2\eta-1)_+}}= C\gamma_n^2 \tau_n^2 =:v\\
  & \sum_{i=1}^n\bE(\tilde \zeta_{i,n}(x)|\cF_{i-1}) 
    \le C\varepsilon\gamma_n \sum_{i=1}^n \lambda_{i-1}^{-\eta} b_i^{-d-1}\le  C\varepsilon\gamma_n n \lambda_n^{-\eta} b_n^{-d-1} 
    = C\gamma_n \varepsilon r_n =:u.
  \end{align*}
  We apply Proposition~\ref{prop:freedmanx} with
  \begin{align*}
  &t:=C_0\sqrt{v\log n}= C_0 \gamma_n \tau_n \sqrt{C\log n}\\
  &\varepsilon 
  := t \left(C r_n\gamma_n \right)^{-1}\,,
  \end{align*}
  for some $C_0\ge 1$. With this specific choice of $\varepsilon$, we remark that $u=t$. In order to compute the factor of $t^2$ in the exponent of
  (\ref{berns}), let us show first that $tm\le  C_0v$ for $n$ large enough. This is equivalent to
  $m^2\log n\le v$, and is true if
  \begin{align}\label{hktlkmy}
    \lim_{n\to \infty}  \frac{\lambda_n^{(2\eta -1)_+}b_n^{d(1-2\eta)_+} \log n}{n\lambda_n^{2\eta}b_n^{d}} =0  \,.
  \end{align}
Since $\gamma_n \ge \frac {C_\gamma} n $,
 it is also true if, 
\[
\lim_{n\to\infty} \frac{n\gamma_n^2 \log n}{\lambda_n b_n^d } =0\,.
\]
The latter is true under Assumption~\ref{hyp:rate}--\ref{hyp:rate:coef}.
Thus,~(\ref{hktlkmy}) is true for $n\ge n_0$ with $n_0$ that depends only on quantities given in assumptions.
It leads to (for $n\ge n_0$):
\[
\max(v,2mu) + \tfrac 23 mt= \max(v,2mt)+\tfrac 23 mt\le 3C_0 v\,.
\]
  We obtain for $n\ge n_0$, a bound on the exponent in (\ref{berns}):
  \begin{align*}
  \frac{t^2}{8 \p{\max(v,2mu)+\tfrac 23 mt}} 
  \ge \frac{C_0^2v\log n}{24C_0v}
  =\frac{C_0}{24}\log n.
  \end{align*}
  By applying (\ref{berns}) and setting $C_0=24q$ for $q\ge 1$, we obtain, recalling $M_n(x) = \sum_{i=1}^n \zeta_{i,n}(x)$,
  \begin{align*}
  &\bP\left(\sup_{\|x\|\le A}|M_n(x)| > 48q \gamma_n \tau_n\sqrt{\log n} \right) 
  \le C'\p{1+ A  \frac{r_n}{\tau_n}}^dn^{-q},~~~~n\ge n_0.
  \end{align*}
  The conclusion follows, taking $C =\max(48q, C')$.
  \end{proof}
\begin{proof}[Proof of Proposition~\ref{prop:mtgl}]
  By choosing $q$ large enough in Lemma~\ref{prop:freedman-appli}, we obtain
  \begin{align*}
  &\sum_n \bP\Big(\sup_{\|x\|\le n^p}|M_n(x)| > t_n\Big) 
  <\infty\\
  &t_n=C\gamma_n q\tau_n\sqrt{\log n}.
  \end{align*}
  For $\omega\in E$ such that $\proba(E) =1$, by Borel-Cantelli's lemma there exists $N(\omega)$ such that
  \begin{align*}
  |M_n(x)| \le t_n<\infty~~~\text{if} ~~n>\max(N(\omega),\|x\|^{1/p}).
  \end{align*}
%   On the other hand
%   \begin{align*}
%   M_n
%   &=\sum_{i=1}^n\psi_{i+1}^n\gamma_i(S_i-S_{i-1})=\sum_{i=1}^{n-1}(\psi_{i+1}^n\gamma_i -\psi_{i+2}^n \gamma_{i+1})S_i+\gamma_nS_n\,.
%   \end{align*}
%   Thus,
%   $$
%   \sup_{\|x\|\le n^p}|M_n(x)|\le \sum_{i=1}^{n-1}\psi_{i+2}^n\gamma_{i+1}\gamma_it_n+ \sum_{i=1}^{n-1}\psi_{i+2}^n\gamma_{i+1}|\tfrac{\gamma_{i}}{\gamma_{i+1}} - 1|t_n + \gamma_n t_n\,.
%   $$
% Remark that, by Assumption~\ref{hyp:rate},
% \[
% |\tfrac{\gamma_{i}}{\gamma_{i+1}} - 1|  \le \tfrac 1C \gamma_i\,.
% \]
%   Point~\ref{pointinf} of Lemma~\ref{lem:rm-determ} with $a_n=\gamma_n$
%   entails $\sum_{i=1}^{n-1}\psi_{i+2}^n\gamma_{i+1}\gamma_i=O(n^{-\alpha}\log n)$. 
  % Thus $|M_n|=O(\tau_n n^{-\alpha}(\log n)^{3/2})$. 
  By Assumption~\ref{hyp:rate}-\ref{hyp:rate:coef},
  $
\lim_{n\to\infty} \frac{t_n^2}{\lambda_n^{2(1-\eta)}}=0
  $,
   which finishes the proof.
\end{proof}
  
\subsubsection{Proof of Proposition~\ref{prop:minZ}}
We define $V_n := {\ovl Z_n -1}$. Then, we obtain
  \[
    V_{n+1} = (1-\gamma_{n+1})V_n + \int \xi_{n+1}
  \]
Iterating the latter, we obtain
\[
V_n = \sum_{i=1}^n\psi_{i+1}^n \gamma_i \int \xi_{i}
\]
We will apply Proposition~\ref{prop:freedmanx} to the sum 
$
    V_n = \sum_{i=1}^n \int \zeta_{i,n}
$, where we recall $\int \zeta_{i,n} =\psi_{i+1}^n \gamma_i \int \xi_i = \psi_{i+1}^n \gamma_i w_i^\eta$.
  We see that:
  \[\begin{split}
    \espcond{\p{\int\xi_{n+1}}^2}{\cF_n}&\le \espcond{w_{n+1}^{2\eta}}{\cF_n}
    =\int f^{2\eta}q_n^{1-2\eta} \,.
  \end{split}
    \]
  If $\eta < \frac 12$, we use Hölder's inequality:
    \[\begin{split}
      \espcond{\p{\int\xi_{n+1}}^2}{\cF_n} &\le \p{\int f}^{2\eta}\p{\int q_n}^{1-2\eta}=
       1 \,.
    \end{split}
    \]
  If $\eta\ge \frac 12$, we use the majoration $q_n\ge\lambda_n q_0\ge \lambda_n c f $ (cf.~(\ref{eq:q}) and Assumption~\ref{hyp:fq}-\ref{hyp:fq:f<q0}):
  \[\begin{split}
    \espcond{\p{\int\xi_{n+1}}^2}{\cF_n} &\le c^{1-2\eta} \lambda_n^{1-2\eta}\,.
  \end{split}
  \]
  Using Lemma~\ref{lem:psigamma}, we obtain the bound
  \[
  \sum_{i=1}^n\espcond{\p{\int \zeta_{i,n}}^2 }{\cF_{i-1}} \le \sum_{i=1}^n C\gamma_n^2 \lambda_{i-1}^{-(2\eta -1)_+}\le Cn\gamma_n^2 \lambda_n^{-(2\eta -1)_+} \,.
  \]
Moreover, $\max_{i = 1\dots n}| \int\zeta_{i,n}| \le C\gamma_n$.
Let $C_0>1$, we apply Proposition~\ref{prop:freedmanx} with $\varepsilon=1$, $A:=1$, $u:=0$, $v := Cn\gamma_n^2 \lambda_n^{-(2\eta-1)_+}$, $m:=C\gamma_n$, and $t := C_0\sqrt{v\log n}$.
We claim that $tm\le v$ for $n$ big enough. This is equivalent to
$C_0 m^2\log n \le v$. Remaking $\gamma_n \ge \tfrac {C_\gamma}n$, the inequality 
  is true, since by Assumption~\ref{hyp:rate}, $n\gamma_n^2\log n \to 0$.
  Consequently, for $n$ big enough
  \[
        \frac{t^2}{8(v+ \tfrac 23 mt} \ge \frac{C_0^2 v\log n}{16 C_0 v} \ge \tfrac {C_0}{16} \log n\,.
  \]
  Consequently we obtain taking $C_0=32$
  \[
  \bP(\abs{V_n} > 32\sqrt{v\log n} ) \le C n^{-2}\,.
  \]
  For $\omega\in E$, such that $\bP(E)=1$, by Borell Cantelli's lemma, there exists $N(\omega)$ such that
  $
  \abs{V_n}\le 32\sqrt{v\log n}
  $, for $n\ge N(\omega)$.
  Since by Assumption~\ref{hyp:rate}--\ref{hyp:rate:coef}, $v\log n\to 0$, Proposition~\ref{prop:minZ} is proven.

% $g_{ n}$ can be rewritten as ($g_0=0$):
% \begin{equation}\label{eq:gtildesum}
%  g_{ n}  =  \sum_{i=1}^n\psi_{i+1}^n \gamma_i    \frac{(1-\lambda_i)^{1-\eta}}{\ovl Z_n^{1-\eta}} \udl Tg_{ i-1}\ast K_{b_i}+\sum_{i=1}^n\psi_{i+1}^n \gamma_i \lambda_i^{1-\eta}  \udl T q_0\ast K_{b_i} +   M_n  
% \end{equation}

\subsection{Proof of Lemma~\ref{lem:gsupv}}
The lemma is a direct consequence of the following lemma.
\begin{lemma}\label{lem:control} Suppose $ f =f_u$, let Assumptions~\ref{hyp:fq},~\ref{hyp:K} and~\ref{hyp:rate} hold true. For all $\alpha>0$ and $p>0$, almost surely, there exists $N_{\alpha,p}\in (0,\infty)$ such that,
  \[
    \min_{f(x)\ge \alpha, \norm{x}\le n^p}\p{ \sum_{i=1}^n\psi_{i+1}^n \gamma_i \lambda_{i-1}^{1-\eta}   \udl T q_0\ast K_{b_i} +   M_n }(x) \ge 0 \,,
  \]
  for every $n\ge N_{\alpha,p}$.
\end{lemma}
\begin{proof}

    % Let fix $\omega\in B$, such that Proposition~\ref{prop:mtgl} holds and $\bP(B)=1$.
  Let $x\in\bR^d$, such that, $f(x)\ge \alpha$. For every $i\ge 1$, $\udl Tq_0 *K_{b_i} (x) \ge c^{1-\eta} \alpha 2^{-\eta}$. 
  Since $(\lambda_n)_{n\ge 0}$ is nonincreasing, 
   $$
  \frac 1{\lambda_{n}^{1-\eta}} \sum_{i=1}^n\psi_{i+1}^n \gamma_i \lambda_{i-1}^{1-\eta}  \udl T q_0\ast K_{b_i}(x) \ge c^{1-\eta}\alpha 2^{-\eta} \sum_{i=1}^n \psi_{i+1}^n\gamma_i \,.
  $$
  By Lemma~\ref{lem:rm-determ} with $a_i=1$ and by Proposition~\ref{prop:mtgl}, almost surely, we obtain  
    \[
 \lim_{n\to\infty}  \frac 1{\lambda_n^{1-\eta}} \min_{f(x)\ge \alpha, \norm{x}\le n^p}\p{ \sum_{i=1}^n\psi_{i+1}^n \gamma_i \lambda_{i-1}^{1-\eta}   \udl T q_0\ast K_{b_i} +   M_n }(x) \ge c^{1-\eta}\alpha2^{-\eta} \,,
  \]
which finishes the proof.
\end{proof}
% We define:
% \[
%     v_n := \sum_{i=1}^n\psi_{i+1}^n \gamma_i    \frac{(1-\lambda_i)^{1-\eta}}{\ovl Z_n^{1-\eta}}\udl Tg_{i-1}\ast K_{b_i}\, .
% \]

\subsection{Proof of Prop~\ref{prop:minorationv}}
In this subsection, we let Assumptions~\ref{hyp:fq},~\ref{hyp:K} and~\ref{hyp:rate} hold true. Moreover, we assume $f=f_u$.

$(v_n)$ can be viewed recursively as:
\begin{equation}\label{eq:vnbecomes}
    v_{n+1} = (1-\gamma_{n+1})v_n + \gamma_{n+1}\frac{(1-\lambda_n)^{1-\eta}}{\ovl Z_n^{1-\eta}} \udl T{g}_{ n} \ast K_{b_{n+1}}\, ,
\end{equation}
with $v_0=0$.

We denote for a real valued function $h$ defines in $\reels^d$:
\begin{equation}\label{eq:notbas}
        \udl h_{\,\varepsilon}(x)  := \inf_{y\in B(x,\varepsilon)} h(y)\, ,
\end{equation}
for all $\varepsilon, x$.
In the case of a sequence of real functions $(h_n)$, we define
\begin{equation}\label{eq:notebas1}
        \udl g_{\, n,\varepsilon}(x)  := \inf_{y\in B(x,\varepsilon)}  g_{n}(y)\, .
\end{equation}

Let $B\subset\Omega$ such that, Proposition~\ref{prop:minZ} and Lemma~\ref{lem:control} hold for every $\omega\in B$ and $\bP(B)=1$.
  We fix $\varepsilon>0, x\in\bR^d, \omega\in B$ in the whole proof and assume that $\udl f_{\,\varepsilon}(x)>\alpha$.
  For any $y\in B(x,\varepsilon)$, we obtain
  \begin{align*}
    \udl T g_n(y) &\ge  \udl T g_n \1_{B(x,\varepsilon')}\ast K_{b_{n+1}} (y)    \\
    & \ge2^{-\eta} \ud{f}{\varepsilon}^{\eta}(x)\ud{g}{n,\varepsilon'}^{1-\eta}(x) (\1 _{B(x,\varepsilon')}\ast K_{b_{n+1}} ) (y)\\
    &\ge  2^{-\eta}\ud{f}{\varepsilon}^{\eta}(x) \ud{g}{n,\varepsilon}^{1-\eta}(x) (\1 _{B(0,\varepsilon')}\ast K_{b_{n+1}} ) (y-x)  \\
    &\ge  2^{-\eta} \tilde c_{\varepsilon} \ud{f}{\varepsilon}^{\eta}(x)\ud{g}{n,\varepsilon}^{1-\eta}(x)\,,
\end{align*}
 where $\tilde c_{\varepsilon}:=\inf_{n\ge 0}\inf_{B(0,\varepsilon)}K_{b_{n+1}}*\1_{B(0,\varepsilon)}$ is a non-negative constant by Lemma~\ref{lem:reste-noyau-bis}. %(i.e. $t=y-x$).
  Going back to~\eqref{eq:vnbecomes}, for every $y\in B(x,\varepsilon)$, we obtain
  \[
      v_{n+1}(y) \ge (1-\gamma_{n+1})v_n(y) + \gamma_{n+1}2^{-\eta}c_{\varepsilon,n} \ud{f}{\varepsilon}^{\eta}(x)\ud{g}{n,\varepsilon}^{1-\eta}(x)  \,,
  \]
  where $c_{\varepsilon,n} := \tilde c_{\varepsilon}\tfrac{(1-\lambda_n)^{1-\eta}}{\ovl Z_n^{1-\eta}}$. By~\eqref{eq:gnsupvn}, we obtain
   \[
      v_{n+1}(y) \ge (1-\gamma_{n+1})v_n(y) + \gamma_{n+1}2^{-\eta}c_{\varepsilon,n} \ud{f}{\varepsilon}^{\eta}(x)\ud{v}{n,\varepsilon}^{1-\eta}(x)  \,,
  \]
for every $n\ge N_{\alpha,p}(\omega)$. Moreover, there exists $\delta>0$ and a constant $N_\delta(\omega)$, such that
\[
\tfrac{(1-\lambda_n)^{1-\eta}}{\ovl Z_n^{1-\eta}} \ge 1-\delta
\]
for every $n\ge N_\delta(\omega)$ by Prop~\ref{prop:minZ}. Consequently, taking the infimum over the ball $B(x,\varepsilon)$, we obtain
   \[
      \ud{v}{n+1,\varepsilon}(x) \ge (1-\gamma_{n+1})\ud{v}{n,\varepsilon}(x) + \gamma_{n+1}(1-\delta)2^{-\eta}c_{\varepsilon} \ud{f}{\varepsilon}^{\eta}(x)\ud{v}{n,\varepsilon}^{1-\eta}(x)  \,,
  \]
  for every $n\ge N_{\alpha,x,\delta}(\omega) := \max(N_{\alpha,x}(\omega), N_\delta(\omega))$. 
  % We define the function: $\iota (x):= (1-\delta)c_\varepsilon \alpha x^{1-\eta}  -x$.
  % Consequently,
  % $$
  % \ud{v}{n+1,\varepsilon}(x) \ge  \ud{v}{n,\varepsilon}(x) + \gamma_{n+1}\iota( \ud{v}{n,\varepsilon}(x) ) 
  % $$
  There exists a constant $c>0$ small enough satisfying
  \[
   (1-\delta)2^{-\eta} c_\varepsilon \alpha c^{1-\eta}  \ge c,\quad  \text{and }\quad \ud{v}{N_{\alpha,x,\delta}(\omega) ,\varepsilon} \ge c
   %,\quad \text{and}\quad \sup_{k\ge N_{\alpha,x,\delta}(\omega)}\gamma_{k} \le C 
   \,.
  \]
  One shows by induction that $\ud{v}{n,\varepsilon}(x) >c$ for every $n\ge N_{\alpha,x,\delta}(\omega)$.
  Taking the limit inferior, we obtain Prop~\ref{prop:minorationv}.

\subsection{Proof of Proposition~\ref{prop:minorationgparf}}
In this subsection, we let Assumptions~\ref{hyp:fq},~\ref{hyp:K} and~\ref{hyp:rate} hold true. Moreover, we assume $f=f_u$.
We also use the notations in~\eqref{eq:notbas} and~\eqref{eq:notebas1}.

\begin{lemma}\label{lem:minorationg}
%Suppose $f=f_u$, let %Assumptions~\ref{hyp:fq},~\ref{hyp:K} and~\ref{hyp:rate} hold true. 
For every $\varepsilon>0$, $x \in\reels^d$, almost surely, we obtain
  \begin{equation}\label{eq:ulb}
    \varliminf \udl g_{\, n,\varepsilon}(x) \ge \udl f_{\, \varepsilon}(x) 
  \end{equation}
\end{lemma}
\begin{proof} 
Let $B\subset\Omega$ such that, Proposition~\ref{prop:minZ} and~\ref{prop:minorationv} holds for every $\omega\in B$, and $\bP(B)=1$.
  We fix $\varepsilon>0,  x\in \bR^d,\omega\in B$ in the whole proof, and assume that $\udl f_{\,2\varepsilon}(x)>0$, 
  because otherwise~(\ref{eq:ulb}) is obvious. 

Since $Z_n\le \bar Z_n$, we obtain for every $y\in\bR^d$,
\[
    q_n(y) \ge \frac{1-\lambda_n}{\bar Z_n}g_n(y) + \lambda_n q_0(y)\ge\frac{1-\lambda_n}{\bar Z_n}g_n(y)\,.
\]
Hence, we obtain $Tq_n(x) \ge c_n Tg_n(y)$, where we defined $c_n := ((1-\lambda_n)/\bar Z_n)^{1-\eta}$. Note that $c_n\to 1$ a.s., by Proposition~\ref{prop:minZ}.
From~\eqref{eq:gn}, for every $y\in \bR^d$, we obtain 
\begin{equation}\label{eq:ming}
    g_{n+1}(y) \ge (1-\gamma_{n+1})g_n(y) + \gamma_{n+1} c_n Tg_n\ast K_{b_{n+1}}(y)   +\gamma_{n+1}\xi_{n+1}(y)\,.
\end{equation}
Let $2\varepsilon\ge \varepsilon'\ge \varepsilon$. Remark that, for $y\in B(x,\varepsilon)$,
\begin{align*}
    Tg_n\ast K_{b_{n+1}}(y) &\ge  Tg_n \1_{B(x,\varepsilon')}\ast K_{b_{n+1}} (y)    \\
    & \ge \ud{f}{\varepsilon'}^{\eta}(x)\ud{g}{n,\varepsilon'}^{1-\eta}(x) (\1 _{B(x,\varepsilon')}\ast K_{b_{n+1}} ) (y)\\
    &\ge \ud{f}{\varepsilon'}^{\eta}(x)\ud{g}{n,\varepsilon'}^{1-\eta}(x) (\1 _{B(0,\varepsilon')}\ast K_{b_{n+1}} ) (y-x)  \\
    &\ge  \tilde c_{\varepsilon',n} \ud{f}{\varepsilon'}^{\eta}(x)\ud{g}{n,\varepsilon'}^{1-\eta}(x)\,,
\end{align*}
  where $\tilde c_{\varepsilon',n}:=\inf_{B(0,\varepsilon')}K_{b_{n+1}}*\1_{B(0,\varepsilon)}$. %(i.e. $t=y-x$).
  Going back to~\eqref{eq:ming}, for every $y\in B(x,\varepsilon)$, we obtain
  \[
      g_{n+1}(y) \ge (1-\gamma_{n+1})g_n(y) + \gamma_{n+1}c_{\varepsilon',n} \ud{f}{\varepsilon'}^{\eta}(x)\ud{g}{n,\varepsilon'}^{1-\eta}(x)  +\gamma_{n+1}\xi_{n+1}(y)\,,
  \]
  where $c_{\varepsilon',n} := c_n\tilde c_{\varepsilon',n}$.
  Iterating the latter equation, for every $y\in B(x,\varepsilon)$, we obtain (since $g_0=0$)
  \[
  g_n(y) \ge \sum_{i=2}^{n}\gamma_{i}\psi_{i+1}^n c_{\varepsilon',i-1} \ud{f}{\varepsilon'}^{\eta}(x)\ud{g}{i-1,\varepsilon'}^{1-\eta}(x) + M_n(y) \,.
  \]
  Taking the infimum over the ball $B(x,\varepsilon)$, we obtain
  \[
    \ud{g}{n,\varepsilon}(x)  \ge \sum_{i=2}^{n}\gamma_{i}\psi_{i+1}^n c_{\varepsilon',i-1} \ud{f}{\varepsilon'}^{\eta}(x)\ud{g}{i-1,\varepsilon'}^{1-\eta}(x) -{\sup_{y\in B(x,\varepsilon)} \abs{M_n(y)} } 
  \]
  Applying Lemma~\ref{lem:rm-determ}, we obtain
  \[
  \varliminf  \ud{g}{n,\varepsilon}(x) \ge \varliminf c_{\varepsilon',n-1} \ud{f}{\varepsilon'}^{\eta}(x)\ud{g}{n-1,\varepsilon'}^{1-\eta}(x) - \varlimsup {\sup_{y\in B(x,\varepsilon)} \abs{M_n(y)} }\,.
 \]
Since $c_n \to 1$, and $\tilde c_{\varepsilon',n} \to 0$ by Lemma~\ref{lem:reste-noyau} and Proposition~\ref{prop:minZ}, we obtain  $\varliminf  c_{\varepsilon',n-1}\ud{g}{n-1,\varepsilon'}^{1-\eta}(x) =  (\varliminf \ud{g}{n-1,\varepsilon'}(x))^{1-\eta}$. Moreover, by Proposition~\ref{prop:mtgl}, $\varlimsup \sup_{y\in B(x,\varepsilon)} \abs{M_n(y)}  = 0$. Hence,  we obtain
\[
    \varliminf \ud{g}{n,\varepsilon}(x)\ge (\ud{f}{\varepsilon'}(x))^{\eta} \varliminf \ud{g}{n,\varepsilon'}(x)^{1-\eta} \,.
\]
  Letting $\varepsilon$ decrease to some value $t<t$ and then $\varepsilon'$ decrease to $t$, 
  we get that $\varliminf \ud{g}{n,t+}(x) =\lim_{s\downarrow t}\varliminf \ud{g}{n,s}(x)$ satisfies
  \begin{align*}
  \varliminf \ud{g}{n,t+}(x)&\ge \udl f_{\,t}^\eta \varliminf \ud{g}{n,t+}(x)^{1-\eta}\,.
  \end{align*}
  By Proposition~\ref{prop:minorationv}, $
  \varliminf \ud{g}{n,t+}(x)  \ge  \varliminf \ud{g}{n,\varepsilon}(x) > 0
$
. Consequently, $ \varliminf \ud{g}{n,t+}(x) \ge \udl f_{\,t}$.
    The latter implies that for any $\varepsilon>s>t$, close enough to $t$, $\varliminf \ud{g}{n,t}(x)\ge\varliminf \ud{g}{n,s+}(x) \ge \udl f_{\,s}$, 
  hence  $\varliminf \ud{g}{n,t}(x)\ge  \udl f_{\,t}$ by continuity 
  of $t\mapsto \udl f_{\,t}$.
  \end{proof}
  \begin{proof}[Proof of Proposition~\ref{prop:minorationgparf}]
% A direct consequence of equation (\ref{eq:gnsupvn}) and Lemma~\ref{lem:minorationv} gives:
% \begin{equation}\label{eq:pr3}
%   \varliminf\udl g_{\,n,\varepsilon}(x)\ge \udl f_{\,\varepsilon} (x)\,,
% \end{equation} for all $x\in\reels^d$ and $\varepsilon>0$.
Suppose that the proposition doesn't hold. Then, there exists a bounded sequence $(x_n)$ and $\delta>0$ such that:
\[
   \p{ g_{ n}(x_n) -f (x_n)} \le -\delta\, .
\]
$(x_n)$ admits a converging subsequence. So, we can suppose without loss of generality that $x_n \to x$.
By continuity of $f$ there exists $\varepsilon>0, n_0\ge 0$ such that:
\[
  f(x_n)\le \udl f_{\,\varepsilon}(x) +\frac \delta2\,, ~~~~~ n\ge n_0\,.
\]
And for $n_1$ big enough:
\[
     g_{n} (x_n)\ge \udl g_{\,n,\varepsilon} (x)\,,~~~~~n\ge n_1\,.
\]
Then, for all $n\ge \max(n_0,n_1)$, we have:
\[
  -\delta\ge\p{  \udl g_{\,n,\varepsilon} (x) -\udl f_{\,\varepsilon}(x)} -\frac \delta2 \,.
\]
Using Lemma~\ref{lem:minorationg}, we obtain $-\delta \ge -\delta/2$, which is a contradiction. 
  \end{proof}
% \begin{proof}[Proof of Corollary~\ref{cor:Zn}]
%   According to Proposition ~\ref{prop:minorationgparf}, we obtain
%   \begin{equation}\label{eq:prop3}
%     \varliminf \inf_{x\in K}\p{g_n(x)-f(x)}\ge 0 \,,
%   \end{equation}
%   for every compact set $K\subset \reels^d$.
%   Let us fix $\varepsilon>0$, by Assumption~\ref{hyp:fq}, there exists a compact set $K_\varepsilon$ such that:
%   \[
%       \int_{K_\varepsilon} f(x)\dr x \ge 1-\varepsilon\, .
%   \]
% Moreover,
%   \[
%     \int_{K_\varepsilon}\p{g_n -f}\ge \abs{K_\varepsilon}\inf_{x\in K_\varepsilon}\p{g_n(x)-f(x)} \,. 
%   \]
%   By~(\ref{eq:prop3}), we obtain
%   \[
%     \varliminf Z_n \ge \varliminf \int_{K_\varepsilon}g_n\ge\int_{K_\varepsilon} f \ge 1-\varepsilon\,,
%   \]
%   for every $\varepsilon>0$, which gives the result.
% \end{proof}

\subsubsection{Proof of Proposition~\ref{prop:majoration}}
In this subsection, we let Assumptions~\ref{hyp:fq},~\ref{hyp:K} and~\ref{hyp:rate} hold true. Moreover, we assume $f=f_u$.

We recall~\eqref{eq:gsum}
\[
  g_{ n}   = \sum_{i=1}^n\psi_{i+1}^n \gamma_i T q_{i-1}\ast K_{b_i}+   M_n  \, .
\]
We define:
\begin{equation}\label{eq:un}
    u_n :=    \sum_{i=1}^n\psi_{i+1}^n \gamma_i    T q_{i-1}\ast K_{b_i}.
\end{equation}
Recursively, we have:
\begin{equation}\label{eq:recun}
    u_{n+1} = (1-\gamma_{n+1})u_n + \gamma_{n+1} Tq_n\ast K_{b_{n+1}}\, ,
\end{equation}
with $u_0=0$.

\begin{lemma}\label{lemupb}  With probability one, for any $\varepsilon>0$ and any $x\in\reels^d$
  \begin{align}\label{upb0}
  &\varlimsup \ovl g_{n,\varepsilon}(x) \le \ovl f_\varepsilon(x).
  \end{align}
  \end{lemma}
  \begin{proof}
  Let $B\subset\Omega$ such that $\bP(B) =1$, and~\ref{eq:Zn} and Proposition~\ref{prop:mtgl} hold for every $\omega\in B$.
  In the whole proof, we fix $\omega\in B$
  From~(\ref{eq:un}), we get for any $x\in\bR^d$
  \begin{align}\label{vnsumxxx}
  u_{n+1} (x)
  %&\le (1-\gamma_{n+1})v_n(x) +\gamma_{n+1} \|K_{b_{n+1}} *  Tg_n\|_\infty
  \le \sum_{i=1}^n\psi_{i+1}^n  \gamma_i\|K_{b_i} *  Tq_{i-1}\|_\infty.
  \end{align}
  Set
  \begin{align*}
  &\theta_n=\max_{\|x\|\le n^p} q_n(x)
  \end{align*}
  and notice that 
  \begin{align*}
  \|K_{b_{n+1}}*Tq_n\|_\infty 
  &=\sup_{x\in\bR^d}\int K_{b_{n+1}}(x-y) f(y)^\eta q_n(y)^{1-\eta}dy \\
  &\le  \|f\|_\infty^\eta  \theta_n^{1-\eta}  
  +\sup_{x\in\bR^d}\int_{\|y\|> n^p} K_{b_n}(x-y) f(y)^\eta  q_{n-1}(y)^{1-\eta}dy.
  \end{align*}
  For any $x\in\bR^d$, using Assumption~\ref{hyp:fq}-\ref{hyp:fq:fbound}:
  \begin{align*}
  \int_{\|y\|> n^p} K_{b_n}(x-y) f(y)^\eta  q_{n-1}(y)^{1-\eta}dy 
  &\le b_n^{-d}\|K\|_\infty\int_{\|y\|> n^p} f(y)^\eta q_{n-1}(y)^{1-\eta}dy \\
  &\le b_n^{-d}\|K\|_\infty \left(\int_{\|y\|> n^p} f(y)dy\right)^\eta~~~\text{(H\"older)}\\
  &\le  \|K\|_\infty b_n^{-d} C_f^\eta n^{-p\eta r}.
  \end{align*}
  Finally, (\ref{vnsumxxx}) yields 
  \begin{align*}
  &u_{n+1}(x) \le
   C\sum_{i=1}^n\psi_{i+1}^n\gamma_i\theta_i^{1-\eta}
   +  C\sum_{i=1}^n\psi_{i+1}^n\gamma_i b_i^{-d}i^{-p\eta r}\,,
  \end{align*}
  and for $\norm{x}\le n^p$, using~\ref{eq:Zn}, there exists a constant $C_0>0$ such that:
  \[
    \theta_n(x)\le C_0 \sup _{\norm y\le n^p} g_n(y) \le C_0\sup_{\norm y\le n^p} {u_n(y)} +C_0\sup_{\|y\|\le n^p} |g_{n}(y)-u_{n}(y)|\,.
  \]
  Then, with $C'=  C{C_0}$ and $g_n - u_n = M_n$, we obtain:
  \begin{align*}
  &\theta_{n+1} \le
   C'\sum_{i=1}^n\psi_{i+1}^n\gamma_i\theta_i^{1-\eta}
   +  C'\sum_{i=1}^n\psi_{i+1}^n\gamma_i b_i^{-d}i^{-p\eta r}+{C_0}\sup_{\|x\|\le (n+1)^p} |M_{n+1}(x)|\,.
  \end{align*}
  By Proposition~\ref{prop:mtgl}, the last term tends to zero.
Lemma~\ref{lem:rm-determ} implies that the second term tends to zero for $p$ large enough. 
  Denote by $r_n$ the sum of these terms, and the first sum by $S_n$, that is: 
  $S_n:= \sum_{i=1}^n\psi_{i+1}^n\gamma_i\theta_i^{1-\eta}$.  
  In order to prove that $\theta_n$ is bounded, it suffices to prove that $S_n$ is bounded. We obtain:
  \begin{align*}
  S_{n+1} 
  &= (1-\gamma_{n+1})S_n +\gamma_{n+1}C' \theta_n^{1-\eta} \\
  &\le (1-\gamma_{n+1})S_n +\gamma_{n+1}C' (C'S_n+r_n)^{1-\eta} \\
  &\le (1-\gamma_{n+1})S_n +\gamma_{n+1}C'^{2-\eta} S_n^{1-\eta}+\gamma_{n+1}C'r_n^{1-\eta} 
  \end{align*}
  Consider  a constant $C_1$ such that 
  \begin{align*}
  & C'^{2-\eta}C_1^{1-\eta}+C'r_n^{1-\eta} \le C_1,~~~~n\ge 1\\
  &S_1\le C_1.
  \end{align*}
  Such a $C_1$ exists, and one shows by induction that $S_n\le C_1$ for all $n$. 
So, $\theta_n$ is bounded as well for all $p$.
  Consider $0<\varepsilon$, taking the sup on $B(x,\varepsilon)$ in (\ref{eq:recun}), we get
  \begin{align*}
  \ovl u_{n+1,\varepsilon}(x) 
  &\le  (1 -\gamma_{n+1}) \ovl u_{n,\varepsilon}(x)
    + \gamma_{n+1} \sup_{y\in B(x,\varepsilon)} K_{b_{n+1}} *  Tq_n(y).
  \end{align*}
  But, for any $\varepsilon'>\varepsilon$
  \begin{align*}
  %\max_{y\in B(x,\varepsilon)} 
  K_{b_{n+1}} *  Tq_n(y)
  &=\int K_{b_n}(y-z) f(z)^{\,\eta} q_{n-1}(z)^{1-\eta}\1_{\|z-x\|\le\varepsilon'}dz\\
  &\phantom{KKK}+\int K_{b_n}(y-z) f(z)^\eta q_{n-1}(z)^{1-\eta}\1_{\|z-x\|>\varepsilon'}dz \\
  &\le \ovl f_{\varepsilon'}(x)^{\eta}\, \ovl q_{n,\varepsilon'}(x)^{1-\eta}  (K_{b_{n+1}} *  \1_{\|.\|\le\varepsilon'})(y-x)\\
  &\phantom{KKK}+\|f\|_\infty\Big(\int K_{b_n}(y-z)^{1/\eta} \1_{\|z-x\|>\varepsilon'}dz\Big)^\eta
  \,.
  \end{align*}
  Using Lemma~\ref{lem:reste-noyau}, the convolution in term first term converges
  to one uniformly on $y\in B(x,\varepsilon)$. 
  As far as the second term is concerned, one has, for $y\in B(x,\varepsilon)$,
  using Assumption~\ref{hyp:K}.(\ref{hyp:K:bound})
  \begin{align*}
  \int K_{b_n}(y-z)^{1/\eta} \1_{\|z-x\|>\varepsilon'}dz
  &\le \int_{\|t\|\ge\varepsilon'-\varepsilon} K_{b_n}(t)^{1/\eta}\dr dt\\
  &\le \int_{\|t\|\ge \varepsilon'-\varepsilon} \Big(b_n^{-d}(\norm{t}/b_n)^{-(r+d)}\Big)^{1/\eta}\dr t\\
  &\le C(\varepsilon,\varepsilon') b_n^{r/\eta}
  \end{align*}
  which tends to zero, since $\int_{\reels^d} \frac 1{1+\norm{t}^{d+\delta}}\dr t$ is bounded for every $\delta>0$. Finally we obtain (we omit the argument $x$)
  \begin{align*}
  \ovl u_{n+1,\varepsilon}
  \le (1 -\gamma_{n+1}) \ovl u_{n,\varepsilon}
    + \gamma_{n+1} \kappa_{n+1}\Big(\ovl f_{\varepsilon'}^{\,\eta} \ovl q_{n,\varepsilon'}^{\,1-\eta}+r_n\Big)
  \end{align*}
  for some sequence $r_n\to 0$ depending on $\varepsilon, \varepsilon'$,
  and $\kappa_n\to 1$.
  The boundness of $\theta_n$ guarantees that the l.h.s. is finite.
  By Lemma~\ref{lem:rm-determ} we have, omitting the argument $x$, 
  \begin{align*}
  \varlimsup \ovl u_{n,\varepsilon} 
  &\le \ovl f_{\varepsilon'}^{\,\eta} \big(\varlimsup \ovl q_{n,\varepsilon'}\big)^{1-\eta}.
  \end{align*}
  Since $\varlimsup \ovl u_{n,\varepsilon}=\varlimsup \ovl q_{n,\varepsilon} =\varlimsup \ovl g_{n,\varepsilon}$ (cf.~\ref{eq:Zn} and Proposition~\ref{prop:mtgl}), 
  the bounded numbers $u_\varepsilon=\varlimsup \ovl g_{n,\varepsilon} $ satisfy
  \begin{align*}
  u_\varepsilon &\le \ovl f_{\varepsilon'}^{\,\eta} u_{\varepsilon'}^{1-\eta}
  \end{align*}
  for any $\varepsilon'>\varepsilon$. The function $\varepsilon\mapsto u_\varepsilon$ is increasing.
  Letting $\varepsilon$ decrease to some value $t$ and then $\varepsilon'$ decrease to $t$, 
  we get that $u_{t+} :=\lim_{s\downarrow t}u_s$ satisfies
  \begin{align*}
  u_{t+} &\le \ovl f_t^{\,\eta} u_{t+}^{1-\eta}
  \end{align*} 
  Hence $u_{t+}\le \ovl f_t$ for all $t$, which implies that $u_t\le \ovl f_t$.
  \end{proof}
  \begin{proof}[Proof of Proposition~\ref{prop:majoration}]
    Suppose that the proposition doesn't hold. Then, there exists a bounded sequence $(x_n)$ and $\delta>0$ such that:
\[
   \p{g_{n}(x_n) -f (x_n)} \ge \delta\, .
\]
$(x_n)$ admits a converging subsequence. Hence, we can suppose without loss of generality that $x_n \to x$.
By continuity of $f$ there exists $\varepsilon>0, n_0\ge 0$ such that:
\[
  f(x_n)\ge \ovl f_{\,\varepsilon}(x)-\frac \delta2\,, ~~~~~ n\ge n_0\,.
\]
And for $n_1$ big enough:
\[
    g_{n} (x_n)\le \ovl g_{\,n,\varepsilon} (x)\,,~~~~~n\ge n_1\,.
\]
Then, for all $n\ge \max(n_0,n_1)$, we have:
\[
  \delta\le\p{  \ovl g_{\,n,\varepsilon} (x) -\ovl f_{\,\varepsilon}(x)} +\frac \delta2 \,.
\]
Now, using Lemma~\ref{lemupb}, we obtain $\delta\le \delta/2$, which finishes the proof.
  \end{proof}
  
\subsection{Proof of Theorem~\ref{th:main}}
  When $f=f_u$, for any compact set $A\subset \reels^d$, we obtain
  \[\begin{split}
    \lim_{n\to\infty} \sup_{x\in A}\abs{g_n(x) - f(x)} & \le \varlimsup \max\p{\sup_{x\in A}\p{g_n(x)-f(x)},-\inf_{x\in A}\p{g_n(x)-f(x)}}\\
    &\le \max\p{\varlimsup\sup_{x\in A}\p{g_n(x)-f(x)},\varlimsup-\inf_{x\in A}\p{g_n(x)-f(x)}}\\
    &\le \max\p{\varlimsup\sup_{x\in A}\p{g_n(x)-f(x)},-\varliminf\inf_{x\in A}\p{g_n(x)-f(x)}}\,.\\
  \end{split}
    \]
  By Propositions~\ref{prop:minorationgparf} and~\ref{prop:majoration}, Theorem~\ref{th:main} is proven.

  When $f\neq f_u$, according to the Step 0, in Sec.~\ref{sec:sketch} the result still holds.

\subsection{Proof of Corollary~\ref{coro:TCL}}
    Let $w_{i,n}$ be defined by
    \[
        w_{i,n} = \frac 1{\sqrt{n}} \p{ w_i h(X_i)- \int fh}\,,
    \]
and check that $\espcond{w_{i,n}}{\cF_{i-1}} = 0$ a.s.
Following from \cite[Corollary 3.1]{hall2014martingale}, it suffices to show
    \begin{align}
       %& \espcond{w_{i,n}}{\cF_{i-1}} = 0\label{eq:TCL1}\\
       &\lim_{n\to\infty} \sum_{i=1}^n\espcond{w_{i,n}^2}{\cF_{i-1}} = \int fh^2 - \p{\int fh}^2\,,\, \text{in probability} \label{eq:TCL2} \\
         &\lim_{n\to\infty} \sum_{i=1}^n\espcond{w_{i,n}^2 \1_{|w_{i,n}|>\varepsilon}}{\cF_{i-1}} = 0 \label{eq:TCL3} \,,\, \text{in probability}.
    \end{align}
Note that 
\[        \espcond{w_{i+1,n}^2}{\cF_i} = \frac 1n \p{\int \frac {f^2}{q_i} h^2 - \p{\int fh}^2}\,.
    \]
    Hence, using the Cesaro theorem, \eqref{eq:TCL2} is a consequence of
    \begin{equation}\label{eq:sufTCL2}
    \lim_{i\to\infty} \int \frac {f^2}{q_i}h^2 = \int fh^2\, ,\, \text{almost surely.}
\end{equation}
Since $\inf_{x\in A} f(x) >0$, by Theorem~\ref{th:main}, there exists $i_0\ge 0$, such that 
$
    \inf_{i\ge i_0,x\in A}q_i(x) > 0\,,
$
for every $\omega\in B$ such that $\bP(B)=1$.  Moreover, for $i\le i_0$,  $q_i \ge \inf_{i\le i_0}\lambda_i>0$.
Hence, we obtain, for all $\omega\in B$,  $\inf_{i\ge 1,x\in A}q_i(x) > 0 $ and therefore $ M:=  \sup_{i\ge 1,x\in A}(f(x) / q_i(x)) <\infty$. Hence, for all $\omega \in B$, the function $ f^2 /q_i h$ is bounded by $ M fh^2  $ which is integrable. In addition, by Theorem \ref{th:main}, for all $\omega\in B$ and all $x\in A$, $q_i(x)\to f(x)$. The Lebesgue dominated convergence theorem implies \eqref{eq:sufTCL2}. 

For the rest of the proof, we fix an arbitrary $\omega \in B$. Remarking that $\int |h| f $ is bounded, when $n $ is large enough ($\varepsilon \sqrt n \ge \int f |h| $) we obtain that   
\[
\begin{split}
 \{x\in \bR^d\,:\,  \abs{\frac 1{\sqrt{n}} \p{ \frac{fh}{q_i}(x)- \int fh}} \ge \varepsilon  \} & = \{x\in \bR^d\,:\,   \frac{fh}{q_i}(x)  \ge \sqrt n \varepsilon  + \int fh \}
 \end{split} 
\]
Taking $  | \int f h | \le \varepsilon \sqrt n / 2  $ yields
\[
\begin{split}
 \{x\in \bR^d\,:\,  \abs{\frac 1{\sqrt{n}} \p{ \frac{fh}{q_i}(x)- \int fh}} \ge \varepsilon  \}& \subset \{x\in A\,:\,fh(x) \ge \tfrac\varepsilon 2 \sqrt{n} q_{i}(x)    \} \\
  & \subset \{ x\in A\,:\,M h(x) \ge \tfrac\varepsilon 2 \sqrt{n}    \}=: A_{n}\,.
\end{split} 
\]
Hence, using Young inequality, we obtain
\[\begin{split}
    \espcond{w_{i+1,n}^2 \1_{\abs{w_{i+1,n}} >\varepsilon}}{\cF_{i}}&\le \frac 2n \int \frac{f^2(x)h^2(x)}{q_i(x)} \1_{A_{n}}(x)\dr x + \frac 2n \p{\int fh}^2 \int \1_{A_{n}}(x)\dr x\\
    & \le \frac{2}n M \int  {f(x)h^2(x)} \1_{A_{n}}(x)\dr x + \frac {2(\int fh)^2}n \int_A \1_{A_{n}}(x)\dr x\,.\\
\end{split}
\]
Finally, we obtain
\[
\begin{split}
\sum_{i=0}^{n-1}\espcond{w_{i+1,n}^2 \1_{\abs{w_{i+1,n}} >\varepsilon}}{\cF_{i}}&\le 2 ( M \vee  (\int fh ) ^2  ) \p{\int  {f(x)h^2(x)} \1_{A_{n}}(x)\dr x +  \int_{A_n} 1 }\,.
\end{split}
\]
By the Lebesgue dominated convergence theorem, the r.h.s. of the above inequality converges to $0$ as $n\to\infty$. Since $\omega$ is arbitrary fixed in $B$,~\eqref{eq:TCL3} holds.  Consequently, the proof is finished.

 \section{Technical results}
\begin{lemma}
   \label{lem:reste-noyau}
  Let $K:\bR^d\to [0,+\infty)$ be a continuous function such that $\int K=1$ and $K(0)>0$. Define $K_b(x)=b^{-d}K(x/b)$ for every $b>0$.
  For every $\varepsilon'> \varepsilon>0$,
  $$
  \lim_{b\downarrow 0}\inf_{B(0,\varepsilon)} K_b  * \1_{B(0,\varepsilon')}  =1\,.
  $$
\end{lemma}
\begin{proof}
  Choose $0<r<\varepsilon'-\varepsilon$.
  For every $\|x\|<\varepsilon$ and every $\|y\|<\tfrac rb$, $\|x-by\|<\varepsilon'$, by the triangular inequality.
  Therefore,
  \begin{align*}
    \inf_{B(0,\varepsilon)}K_b  * \1_{B(0,\varepsilon')} &=  \inf_{\|x\|<\varepsilon} \int \1_{\|x-by\|<\varepsilon'}K(y)dy\ \ge  \int_{\|y\|<\frac r b}K(y)dy\,.
  \end{align*}
  Letting $b$ converge to zero, the monotone convergence theorem implies that
  the righthand side of the above inequality converges to one. 
\end{proof}
\begin{lemma}
  \label{lem:reste-noyau-bis}
  Let $K:\bR^d\to [0,+\infty)$ be a continuous function such that $\int_{\bR^d}K(x)dx=1$ and $K(0)>0$. Define $K_b(x)=b^{-d}K(x/b)$.
There exists $c>0$ such that for every $\varepsilon>0$,
  $$
  \varliminf_{b\downarrow 0} \inf_{B(0,\varepsilon)}K_b *  \1_{B(0,\varepsilon)} \ge c\,.
  $$
\end{lemma}
\begin{proof}
By continuity of $K$, there exists $\delta>0$ such that
  $K(x)\ge K(0)/2$ for every $x\in B(0,\delta)$. Denoting by $\leb_d$ the Lebesgue measure on $\bR^d$,
\begin{align*}
\kappa_b(\varepsilon,\varepsilon) 
&= \inf_{y\in B(0,\varepsilon)}\int_{\bR^d} \1_{B(0,\varepsilon)}(y-bx)K(x)dx\\
&\ge \inf_{y\in B(0,\varepsilon)}\frac{K(0)}2\leb_d(B(y/b,\varepsilon/b)\cap B(0,\delta))\\
&= \frac{K(0)}2\leb_d(B(\varepsilon \vec{v}/b,\varepsilon/b)\cap B(0,\delta))\,
\end{align*}
where $\vec{v}$ denotes any unit norm vector of $\bR^d$. The sequence of sets $(\varepsilon/b)B( \vec{v},1)$ is increasing
  and converges to the half space $E:=\{x:\ps{\vec v,x}\ge 0\}$ as $b\downarrow 0$.
Passing to the limit, the result follows by setting $c:=(K(0)/2)\leb_b(H\cap B(0,\delta))$.  
\end{proof}
\begin{lemma}\label{lem:psigamma}

  Let $(\psi_k^n)_{n\ge, k\ge 1}$ be defined by~(\ref{eq:psi}) and $\psi_n^{n-1}=1$.
    Let $(\gamma_n)_{n\ge 1}$ be a sequence satisfying Assumption~\ref{hyp:rate}--~\ref{hyp:rate:gamma}. Then, there exists a constant $C>0$ satisfying
    \[
    \psi_{n+1}^n\gamma_i \le C \gamma_n\,,
    \]
    for every $1\le i \le n$.
\end{lemma}
\begin{proof}

        Let $n_0$ be the rank given by Assumption~\ref{hyp:rate}--\ref{hyp:rate:gamma}, such that $\gamma_i -\gamma_{i+1}\le \gamma_i\gamma_{i+1}$ for every $i\ge n_0$.
        Let $ n \ge n_0$.
 The result holds when $i=n$. 
    Now, suppose that for a given $n_0+1 \le i_0\le n$, 
        \[
    \psi_{i_0+1}^n\gamma_{i_0} \le \gamma_n\,.
    \]
    We obtain
    \[
    \begin{split}
            \psi_{i_0}^n\gamma_{i_0-1} &= (1-\gamma_{i_0})(\gamma_{i_0} + (\gamma_{i_0-1} - \gamma_{i_0}) )\psi_{i_0+1}^n     \\&\le \gamma_n{(1-\gamma_{i_0})} + \psi_{i_0+1}^n (\gamma_{i_0-1} -\gamma_{i_0})(1-\gamma_{i_0})\\
            &\le \gamma_n(1-\gamma_{i_0})(1+\gamma_{i_0-1})\\
            &\le \gamma_n(1-\gamma_{i_0}\gamma_{i_0-1} + (\gamma_{i_0-1} -\gamma_{i_0}))\\
            &\le \gamma_n\,.
    \end{split}
    \]
    Iterating the latter result, $\psi_{i+1}^n\gamma_i \le \gamma_n$ for every, $n\ge i\ge n_0$. For a given $i<n_0$,
    \[
    \psi_{i+1}^n\gamma_i = \psi_{n_0+1}^n \gamma_{n_0} \frac{\psi_{i+1}^n \gamma_i}{\psi_{n_0+1}^n \gamma_{n_0}}\le \gamma_n \frac{\gamma_1}{\gamma_{n_0}}\,.
    \]
    Hence, Lemma~\ref{lem:psigamma} is proven.
\end{proof}

  \begin{lemma}
    \label{lem:rm-determ}
  Consider a real sequence $(a_n)_{n\ge 1}$, and let $(\gamma_n)_{n\ge 1}$ be a positive sequence converging to zero, and such that
  $\sum_n\gamma_n=+\infty$. Let $(\psi_k^n)_{n\ge k,\ge 1}$ be defined by~(\ref{eq:psi}) and $\psi_n^{n-1}=1$.
  For any $s_0\in \mathbb R$, the sequence $(s_n)_{n\ge 0}$ given by:
  \begin{align*}
  s_n&=(1-\gamma_n)s_{n-1} +  \gamma_n a_n, %~~~s_0=s_0
  \end{align*}
  satisfies for every $n\ge 1$:
  \begin{align}\label{sneq}
  s_n =  \psi_{1}^ns_0+\sum_{i=1}^n \psi_{i+1}^n\gamma_i a_i\,.
  \end{align}
  In addition,
  $\varliminf a_n\le \varliminf s_n \le \varlimsup s_n\le \varlimsup s_n\,.$
  % \item \label{pointinf} We assume here that $\gamma_n\ge \frac1n$ for $n$ large enough.\\
  % If $a_n=O( n^{-\tau})$ for some $0\le\tau< 1$ then $s_n= O(n^{-\tau}).$ \\
  %  If  $a_n=O( n^{-1})$ then $s_n= O(n^{-1}\log n).$
  % \item \label{pointsup} If $a_n^{-1}=O( n^{\tau})$ for some $0\le\tau< 1$ then $s_n^{-1}= O(n^{\tau}).$ 

  \end{lemma}
  \begin{proof}An elementary induction shows that for any $k>0$
  \begin{align*}
  s_n =  \psi_{k+1}^ns_k+\sum_{i=k+1}^n \psi_{i+1}^n\gamma_i a_i\,~~~~n>k.
  \end{align*}
  In particular $\sum_{i=k}^n \psi_{i+1}^n\gamma_i=1-\psi_k^n$, and we get
  \begin{align*}
  \varlimsup_{n>k}s_n \le 0+\sup_{i>k}a_i.
  \end{align*}
  This proves that $\varlimsup s_n\le\varlimsup a_n$. The inequality $\varliminf s_n\ge\varliminf a_n$
  is proved similarly.
   \end{proof}

    We state here Theorem~19 of~\cite{delyon2021safe}:
  \begin{prop}\label{prop:freedmanx}
  Let $(\Omega, \mathcal F , (\mathcal F _j)_{j\ge 1}, \mathbb P)$ be a filtered space. 
  Let $(\xi_j)_{j\ge 1} $ be a sequence of real valued stochastic processes defined on $\mathbb R^d$, 
  adapted to $(\mathcal F_j)_{j\ge 1}$, such that for any $x\in\mathbb R^d$,
  \begin{align*}
  &\mathbb E [\xi_j(x)| \mathcal F _{j-1}] = 0,\quad \text{for all } j\ge 1 .
  \end{align*}
  Consider $\varepsilon>0$ and let  $(\tilde \xi_j)_{j\ge 1} $ be another $(\mathcal F_j)_{j\ge 1}$-adapted 
  sequence of non-negative stochastic processes defined on $\mathbb R^d$ such that for all $j\ge 1$ 
  and $x\in\mathbb R^d$
  \begin{align}\label{freedlip}
  &\sup_{\|y\|\le\varepsilon}|\xi_j(x+y)-\xi_j(x)|\le \tilde \xi_j(x) .
  \end{align}
  Let $n\ge 1$ and assume that for some $A\ge 0$, 
  there exist $m, v, u \in \mathbb R^{+}$ such that  for all $\omega\in\Omega$ and $\|x\|\le A$,	
  \begin{align}
  & \max_{j=1,\ldots, n} |\xi_j(x)|\le m\label{thmm}\\
  &\sum_{j=1} ^ n \mathbb E\big[\xi_j(x)^2|\mathcal F_{j-1}]\le v\\
  &\sum_{j=1} ^ n \mathbb E\big[\tilde \xi_j(x)|\mathcal F_{j-1}\big]\le u .
  \end{align}
  Then, for all $t\ge 0$,
  \begin{align}\label{berns}
  &\mathbb P\Big(\sup_{\|x\|\le A}\big|\sum_{j=1} ^ n \xi_j(x)\big|>t+u\Big)
  \le 4(1+ {2A}\varepsilon^{-1} )^d \exp\left(-\frac{t^2}{8(\tilde v+\tfrac23 mt)} \right) ,
  \end{align}
  with $\tilde v=\max(v,2mu)$.
  \end{prop}
  \end{appendix}

\end{document}